\documentclass[final,onefignum,onetabnum]{siamart190516}

\usepackage{amsmath,amssymb,amsopn}
\usepackage{graphicx}
\usepackage{color}
\usepackage[T1]{fontenc}
\usepackage{booktabs}

\usepackage{algorithm}
\usepackage{algpseudocode} 

\algrenewcommand{\algorithmiccomment}[2][.6\linewidth]{\leavevmode\hfill\makebox[#1][l]{$\triangleright$ #2}}
\algnewcommand{\algorithmicgoto}{\textbf{goto}}%
\algnewcommand{\Goto}[1]{\algorithmicgoto~\ref{#1}}%

\newcommand{\COMMENT}[2][.47\linewidth]{%
  \leavevmode\hfill\makebox[#1][l]{$\triangleright$~#2}}

\usepackage{pgfplots}
\usepackage{pgfplotstable}
\usepgfplotslibrary{groupplots}
\usepgfplotslibrary{colorbrewer}
\usetikzlibrary{pgfplots.groupplots}
\usetikzlibrary{arrows,intersections}

\newsiamthm{thm}{Theorem}
\newsiamthm{cor}{Corollary}
\newsiamthm{lem}{Lemma}
\newsiamremark{rem}{Remark}

\usepackage{doi}

\newcommand{\1}{\xi}

\newcommand{\bits}{\mathcal B}
\renewcommand{\Re}{\mathbb{R}}
\newcommand{\I}{N}
\newcommand{\ead}{\dot{\mu}}
\newcommand{\ea}{\mu}
\newcommand{\bet}{\beta}
\newcommand{\cratio}{\vartheta}

\newcommand{\rh}{\rho}
\newcommand{\A}{\psi}
\newcommand{\kap}{\underline{\kappa}(A)}
\newcommand{\kapj}{\underline{\kappa}(A_j)}
\newcommand{\kapjm}{\underline{\kappa}(A_{j - 1})}

\newcommand{\m}{M}

\newcommand{\ir}{$\mathcal{IR}$}

\newcommand{\V}{$\mathcal{V}$}
\newcommand{\fmg}{$\mathcal{FMG}$}
\newcommand{\pfmg}{$\mathcal{PFMG}$}

\newcommand{\alp}{\dot{m}_P^+}
\newcommand{\alpnodot}{m_P^+}

\newcommand{\alabar}{{\bar m}_A^+}
\newcommand{\aladot}{{\dot{m}}_A^+}

\newcommand{\cc}{\sigma}
\newcommand{\ewe}{\boldsymbol{\varepsilon}}
\newcommand{\ewed}{\boldsymbol{\dot{\varepsilon}}}
\newcommand{\eweb}{\boldsymbol{\bar{\varepsilon}}}
\newcommand{\eweq}{\boldsymbol{\check{\varepsilon}}}
\newcommand{\rinner}{\rho}

\newcommand{\rvstar}{\rho_{v}^*}

\newcommand{\rv}{\rho_{v}}

\newcommand{\pvar}{\tau}
\newcommand{\pvarq}{{\check \tau}}
\newcommand{\pvarb}{{\bar \pvar}}
\newcommand{\pvard}{{\dot{\pvar}}}

\newcommand{\pir}{\delta_{\rho_{ir}}}
\newcommand{\thresh}{\chi}

\newcommand{\pv}{\delta_{\rho_{v}}}

\newcommand{\gmm}{\gamma}
\newcommand{\sol}{u}
\newcommand{\trialfun}{u}
\newcommand{\testfun}{v}
\newcommand{\hot}{\varphi}

\newcommand{\emm}{m}
\newcommand{\bigO}{\mathcal{O}}
\newcommand{\fespace}{\mathcal{U}}

\definecolor{lo}{HTML}{008000}
\definecolor{med}{HTML}{0000FE}
\definecolor{hi}{HTML}{FB0106}

\DeclareMathOperator{\fl}{fl}

\headers{Discretization-error-accurate multigrid solvers}{R. Tamstorf, J. Benzaken, and S. F. McCormick}

\title{Discretization-error-accurate mixed-precision multigrid solvers\thanks{Submitted to the editors June 29, 2020.}}

\author{Rasmus Tamstorf\thanks{Walt Disney Animation Studios, Burbank, CA
    (\email{Rasmus.Tamstorf@disneyanimation.com},\newline\email{Joseph.Benzaken@disneyanimation.com}).}
\and Joseph Benzaken\footnotemark[3]
\and Stephen F. McCormick\thanks{University of Colorado at Boulder, Boulder, CO
  (\email{stephen.mccormick@colorado.edu}).}
}

\ifpdf
\definecolor{darkblue}{rgb}{0.0,0.0,0.3}
\hypersetup{
  pdftitle={Discretization-error-accurate mixed-precision multigrid solvers},
  pdfauthor={R. Tamstorf, J. Benzaken, and S. F. McCormick}
  linkcolor=darkblue,
  citecolor=darkblue,
  filecolor=black,
  urlcolor=darkblue,
  anchorcolor=darkblue
}
\fi

\begin{document}

\maketitle

\begin{abstract}
  This paper builds on the algebraic theory in the companion paper {\em [Algebraic Error Analysis for Mixed-Precision Multigrid Solvers, submitted to SISC]} to obtain discretization-error-accurate solutions for linear elliptic partial differential equations (PDEs) by {\em mixed-precision} multigrid solvers. It is often assumed that the achievable accuracy is limited by discretization or algebraic errors. On the contrary, we show that the quantization error incurred by simply storing the matrix in any fixed precision quickly begins to dominate the total error as the discretization is refined. We extend the existing theory to account for these quantization errors and use the resulting bounds to guide the choice of four different precision levels in order to balance quantization, algebraic, and discretization errors in the progressive-precision scheme proposed in the companion paper. A remarkable result is that while iterative refinement is susceptible to quantization errors during the residual and update computation, the V-cycle used to compute the correction in each iteration is much more resilient, and continues to work if the system matrices in the hierarchy become indefinite due to quantization. As a result, the V-cycle only requires relatively few bits of precision per level. Based on our findings, we outline a simple way to implement a progressive precision FMG solver with minimal overhead, and demonstrate as an example that the one dimensional biharmonic equation can be solved reliably to any desired accuracy using just a few V-cycles when the underlying smoother works well. In the process we also confirm many of the theoretical results numerically. 
\end{abstract}



\section{Introduction}
The {\em abstract} theory in \cite{McCormick2020} analyzes rounding-error effects of {\em algebraic} operations in mixed- and progressive-precision multigrid solvers. The main focus here is applying this analysis to solvers for discretized linear elliptic partial differential equations (PDEs), with the goal of obtaining accuracy on the order of the discretization error in the energy norm. Achieving such algebraic accuracy can often be done by {\em full multigrid (FMG)} at optimal cost (i.e., comparable to a few matrix multiplies on the finest level), but the approximation property that FMG relies on makes it especially sensitive to rounding errors. One of our main goals is to analyze this sensitivity in order to understand how to achieve optimal results.

To this end, first note that existing rounding-error analyses of linear solvers (e.g., \cite{Carson2017,Carson2018,Higham2019,McCormick2020}) typically assume that the target matrix is exact. In practice, this assumption is rarely satisfied since forming the matrix itself is subject to rounding errors. We extend the theory in \cite{McCormick2020} to include errors due to simply storing the linear system in finite precision. Referring to it as {\em quantization} error, we show that its effect grows much faster under mesh refinement than that of algebraic errors. Our analysis is purely algebraic in nature and therefore applies to linear systems regardless of their origin. It also applies to matrix-free methods because quantization happens regardless of whether the result is stored in main memory or just in a register. Even so, it should be emphasized that the quantization error is the smallest possible error one can consider for the formation process, so in some sense this still represents an optimistic analysis. In particular, we do not consider errors due to numerical quadrature during assembly of the linear system. 

To make the theory concrete, we consider classical finite element discretizations of PDEs, which facilitates quantifying the discretization error and comparing it to quantization and algebraic errors. This comparison in turn allows us to explore the optimal relationship between the different precision levels introduced in the progressive-precision multigrid solver in \cite{McCormick2020}. At first, it might appear as if the quantization error limits the benefit of using three precisions in iterative refinement. However, the benefit remains as long as the various precision levels are chosen carefully. Furthermore, because the inner solver only needs to reduce the residual in the iterative refinement scheme by a small amount, we prove that it is \emph{not} necessary to require that the system matrix remains positive definite when rounded to the lowest precision. Avoiding this requirement allows us to use very low precision for the inner solver where most of the computations are performed. By comparison, \cite{Higham2019} also uses low precision for the inner solver, but they require an unknown perturbation to be added to the low precision matrix to recover positive definiteness. 



We begin in the next section by introducing nomenclature for the different errors involved in our analyses. While we assume that the reader is familiar with \cite{McCormick2020}, Section~\ref{sec:fp} summarizes its essential definitions and theoretical estimates for completeness. In Section~\ref{galerkin} we consider the effects on the V-cycle of quantizing the multigrid components and computing them by the {\em Galerkin condition} \cite{tutorial}. The effects on FMG of quantizing the multigrid components are analyzed in Section~\ref{coarsening}. Section~\ref{sec:precision} brings the theory together to establish the progression requirements for the different precision levels studied throughout the paper. Section~\ref{sec:modelproblem} introduces a simple model problem based on the one dimensional biharmonic equation. This problem is then studied in Section~\ref{sec:results} with various mesh sizes and approximation orders. We illustrate the behavior of a standard V-cycle and also demonstrate that the problem can be solved reliably to any accuracy by progressive precision {\fmg } when the precision levels are chosen appropriately. We end the paper with some concluding remarks in the last section.

\section{Error definitions}
\label{sec:errordefs}
The numerical solution of PDEs in finite precision involves several different error sources. This section introduces notation and vocabulary to distinguish these basic types of errors. 

Consider a linear PDE of the form $\mathcal{L}u = f$ subject to some boundary conditions, with source term $f$ and exact solution $u$. Assume that $A_h x_h = b_h$ represents its discretization on a regular grid of element size $h$ by the Galerkin finite element method, where $A_h, x_h,$ and $b_h$ are exact. Assume also that $x_h = (x_{h,i})$, where the $x_{h,i}$ are the coefficients corresponding to the basis functions $\phi_{h,i}$ so that the finite element solution for grid $h$ is the function $u_h = \sum_i x_{h,i} \phi_{h,i}.$ The discretization error in the energy or $\mathcal{L}$ norm for grid  $h$ is then represented by $e_{\textrm{disc}} = \| u_h - u \|_{\mathcal{L}}$. In practice this is computed through quadrature using the bilinear form for the PDE.

Simply rounding the coefficients $x_{h,i}$ to $\bits$ bits is denoted by $\fl(x_h)$ and the corresponding continuous solution (with a slight abuse of notation) by $\fl(u_h) = \sum_i \fl(x)_{h,i}\phi_{h,i}$. Given this, the {\em floating-point error} due to representing the exact solution at grid level $h$ in $\bits$ bits of precision is denoted by $e_{\fl} = \| \fl(u_h) - u_h \|_{\mathcal{L}} = \|\fl(x_h) - x_h\|_{A_h}.$ This is the best level of energy error we can expect to obtain in finite precision.

Since computation with matrices and source terms require that they too be represented in working precision, we let $A_h$ and $b_h$ rounded to $\bits$ bits be denoted by $\check{A}_h$ and $\check{b}_h$, respectively. We also use the ha\v{c}ek diacritical mark for any quantities derived from $\check{A}_h$ and $\check{b}_h$. For example, $\check{x}_h$ represents the \emph{exact} solution of $\check{A}_h x = \check{b}_h$. Note that in general $\check{x}_h \ne \fl(x_h)$ because $\check{x}_h$ is based on $\check{A}_h$ while $\fl(x_h)$ is based on $A_h$ and then rounded to $\bits$ bits. In addition to ha\v{c}eks, we use tilde to denote values \emph{computed} from $\check{A}_h$ and $\check{b}_h$ and superscripts in parentheses for iterations. For example, for the numerical solution of $\check{A}_h x = \check{b}_h$, we denote the $i^\textrm{th}$ iterate by $\tilde{x}_h^{(i)}$ and the fully converged algebraic solution by $\tilde{x}_h^{(\infty)}$.

Corresponding to the vectors $\check{x}_h$, $\tilde{x}_h^{(i)}$, and $\tilde{x}_h^{(\infty)}$ are the functions $\check{u}_h$, $\tilde{u}_h^{(i)}$, and $\tilde{u}_h^{(\infty)}$, which allow us to write the {\em quantization error} as $e_{\textrm{quant}} = \| \check{u}_h - u_h \|_{\mathcal{L}}$ and the {\em algebraic error} as $e_{\textrm{alg}} = \|\tilde{u}_h^{(i)} - \check{u}_h \|_{\mathcal{L}} = \|\tilde{x}_h^{(i)} - \check{x}_h \|_{A_h}$. Note that $e_{\textrm{alg}} \ne \| \tilde{x}_h - \check{x}_h \|_{\check{A}_h}$ in general. Finally, note that while a multigrid algorithm would presumably converge to $(\check{A}_h)^{-1} \check{b}_h$ in infinite precision, it is generally limited from doing so in finite precision. Accordingly, we decompose the algebraic error into {\em iteration error} $e_{\textrm{iter}} = \|\tilde{u}_h^{(i)} - \tilde{u}_h^{(\infty)} \|_{\mathcal{L}}$ and {\em rounding error} $e_{\textrm{round}} = \|\tilde{u}_h^{(\infty)} - \check{u}_h\|_{\mathcal{L}}$. It might be argued that quantization error is also a kind of rounding error, but for purposes of this paper we will consider it separately.

Combined, these definitions allow us to write the total error after $i$ iterations as 
\begin{align}
e_{total}^{(i)} &= \| \tilde{u}_h^{(i)} - u \|_{\mathcal{L}} \nonumber\\
  &= \| \tilde{u}_h^{(i)} - \tilde{u}_h^{(\infty)} + \tilde{u}_h^{(\infty)} - \check{u}_h + \check{u}_h - u_h + u_h - u \|_{\mathcal{L}} \nonumber\\
&\leq \underbrace{\| \tilde{u}_h^{(i)} - \tilde{u}_h^{(\infty)} \|_{\mathcal{L}}}_{e_{\textrm{iter}}} + \underbrace{ \| \tilde{u}_h^{(\infty)} - \check{u}_h\|_{\mathcal{L}} }_{ e_{\textrm{round}} } + \underbrace{ \|\check{u}_h - u_h\|_{\mathcal{L}} }_{ e_{\textrm{quant}} } + \underbrace{ \|u_h - u \|_{\mathcal{L}} }_{ e_{\textrm{disc}} }.
\label{errdecomp}
\end{align}

In summary, quantities without tildes are exact (in infinite-precision arithmetic), those with tildes are computed,  those with ha\v{c}eks are based on the quantized versions of $A_h$ and $b_h$, and those with subscript $h$ have been discretized on a grid with element size $h$. We assume that there are no errors in computing $\check{A}_h$ aside from the quantization itself. Thus, any numerical integration used to compute $A_h$ must be sufficiently accurate and the algebraic error associated with evaluating functions at the quadrature points and summing must be insignificant. 

\section{Existing Theory}
\label{sec:fp}
This section summarizes the notation, conventions, and theory of \cite{McCormick2020}. Initially, we consider three floating point environments: ``standard'' precision with unit roundoff $\ewe$, ``high'' precision with unit roundoff $\eweb$, and ``low'' precision with unit roundoff $\ewed$. We also refer to these as $\ewe$-, $\eweb$-, and $\ewed$-precision respectively. While it is only formally assumed that ${\eweb} \le \ewe \le \ewed$, we address the choice of these precision levels in Section~\ref{sec:precision}.

The theory uses variables and expressions for exact quantities, with $\delta$'s added and estimated to represent quantities computed in finite precision. Thus, for $A \in \Re^{n \times n}$  symmetric positive definite (SPD) with at most $m_A$ nonzeros per row and $b \in \Re^n$, then  $A^{-1}b + \delta$ denotes a {\em computed} approximate solution with error $\delta$ of
\begin{equation}
Ax = b,
\label{axb}
\end{equation}
and $Ax - b + \delta_r$ denotes its {\em computed} residual with error $\delta_r$.

In what follows, we let $|z|$ for $z \in \Re^n$ denote the vector of the absolute values of $z$, and similarly for matrices. Inequalities and equalities between vectors and matrices are defined componentwise. Using $\| \cdot \|$ to denote the Euclidean norm for a vector and its induced matrix norm (together with the Euclidean inner product $\langle \cdot, \cdot \rangle$), we frequently use the fact that $\||z|\| = \|z\|$ for any vector $z$ (although this is not generally true for matrices).
Our rounding-error estimates are in terms of the discrete {\em energy} norm $\| \cdot \|_A$ defined by $\|x\|_A = \|A^{\frac{1}{2}}x\|, x \in \Re^n$. Following the usual convention in rounding-error analyses, we assume that $A$ and $b$ in (\ref{axb}) are exact. To rein in the complexity of our estimates, we take this assumption further by assuming exactness of all of the multigrid components: the intergrid transfer, the system matrix, and the right-hand side on all levels. However, we consider the effects of quantization of these components in Section~\ref{coarsening} because they are used specifically within the multigrid solvers. Notation used below includes
\begin{align*}
\kappa(A)&=\|A\| \cdot \|A^{-1}\|, \, \A = \||A|\|, \, \kap = \A \|A^{-1}\|, \alabar = \frac{m_A + 1}{1 - (m_A + 1){\eweb}}, \\
\dot{m}_A &= \frac{m_A }{1 - m_A{\ewed}}, \pvard = \kappa^\frac{1}{2}(A)\ewed, \pvar = \kappa^\frac{1}{2}(A)\ewe, \pvarb = \kappa(A){\eweb},  \gmm = \frac{\kappa^{\frac{1}{2}}(A) + \kap}{\kappa(A)}.
\end{align*}

The mixed-precision approach analyzed theoretically in \cite{McCormick2020} uses iterative refinement as the outer loop and a generic approximate linear solver as the inner loop. The pseudocode for iterative refinement (\ir), is given in Algorithm \ref{alg-ir} below. The floating-point operations in {\ir } use all three precisions. The full residual $r$ between successive calls to the inner solver is evaluated in $\eweb$-precision (red font), while the inner solver uses $\ewed$-precision (green font). All other operations use $\ewe$-precision (blue font).

\noindent
\begin{center}
\begin{minipage}{0.8\linewidth}
\begin{algorithm}[H]
\footnotesize
\caption{$\texttt{Iterative Refinement (\ir)}$}
\label{alg-ir}
\begin{algorithmic}[1]
      \Require A, b, x initial guess, tol $> 0$ convergence tolerance.
      \State \textcolor{med}{$r \leftarrow$}\textcolor{hi}{$ Ax - b$}\Comment{\textcolor{hi}{Compute {\ir } Residual} and \textcolor{med}{Round}}\label{ir-begin}
      \If{\textcolor{med}{$\|r\| <$ tol}}
      \State\Return $x$\Comment{\textcolor{med}{Return Solution of $Ax = b$}}
      \EndIf
      \State \textcolor{med}{$y \leftarrow$} \textcolor{lo}{$\texttt{InnerSolve}(A, r)$}\Comment{\textcolor{lo}{Compute Approximate Solution of $Ay = r$}}
      \State \textcolor{med}{$x \leftarrow x - y$}\Comment{\textcolor{med}{Update Approximate Solution of $Ax=b$}}
      \State\Goto{ir-begin}
\end{algorithmic}
\end{algorithm}
\end{minipage}
\end{center}
\vspace{1em}

\begin{thm}{\em \ir.  \cite{McCormick2020}}
Let $x^{(i)}$ be the iterate at the start of the $i^\textrm{th}$ cycle of {\ir } and $r = Ax^{(i)} - b$ its residual computed in ${\eweb}$-precision and rounded to $\ewe$-precision. Suppose that a $\rinner < 1$ exists such that, for any $r \in \Re^n$, the solver used in the inner loop of Algorithm~\ref{alg-ir} (line 5) is guaranteed to compute a correction $y$ that satisfies
\[
\|y - A^{-1} r\|_A \le \rinner \|A^{-1} r\|_A .
\]
Then $x^{(i+1)}$ approximates the solution $A^{-1} b$ of (\ref{axb}) with the {\em relative} error bound
\begin{equation}
\frac{\|x^{(i+1)} - A^{-1}b\|_A}{\|A^{-1}b\|_A} \le \rho_{ir} \frac{\|x^{(i)} - A^{-1}b\|_A }{\|A^{-1}b\|_A} + \thresh,
\quad \rho_{ir} = \rinner + \pir,
\label{ec}
\end{equation}
where
\begin{equation}
\pir = \frac{(1 + 2\rinner) \pvar + \gmm(1 + \rinner)(1 + \ewe)\alabar \pvarb}{1 - \pvar},
\quad \thresh = \frac{\pvar + \gmm(1 + \rinner)(1 + \ewe) \alabar\pvarb}{1 - \pvar}.
\label{chi}
\end{equation}
If $\rinner + \pir < 1$, then the error after $\I \ge 1$ {\ir } cycles with initial guess $x^{(0)}$ satisfies
\begin{equation}
\frac{\|x^{(\I)} - A^{-1}b\|_A}{\|A^{-1}b\|_A} \le (\rinner + \pir)^\I \frac{\|x^{(0)} - A^{-1}b\|_A}{\|A^{-1}b\|_A} + \frac{\thresh }{1 -(\rinner + \pir)} .
\label{Kbnd}
\end{equation}

\label{theorem2}
\end{thm}

For the inner solver, we use one V$(1, 0)$-cycle, with pseudocode {\V } given in Algorithm~\ref{alg-mg} below. {\V } uses a nested hierarchy of $\ell$ grids from the coarsest $j = 1$ to the finest $j = \ell$, $1 \le j \le \ell$. It begins on the finest grid and proceeds down to the coarsest grid, with one relaxation sweep on each level along the way.  Each level is equipped with a system matrix $A_j$, with $A_\ell = A$. Assume that relaxation on grid $j$ applied to $A_jy_j = r_j$ is the stationary linear iteration $y_j \leftarrow y_j - \m_j (A_jy_j - b_j)$, where $M_j$ roughly approximates $A_j^{-1}$. Let $P_j$ denote the interpolation matrix that maps from grid $j - 1$ to grid $j$ with at most $m_{P_j}$ nonzeros per row or column. Let $P_1 = 0$ for the coarsest grid, which involves just one relaxation sweep and no further coarsening. Assume further that the  Galerkin condition is exactly satisfied on all coarse levels: $A_{j - 1} = P_j^tA_jP_j, \, 2 \le j \le \ell.$ (See Section~\ref{galerkin} for analysis of the rounding-error effects when this relationship is used to compute the coarse-grid matrices in finite precision.) All computations in {\V } are performed in low $\ewed$-precision as shown in green font in the  pseudocode. Accordingly, since the input right-hand side (RHS) may be in higher precision, the cycle is initialized with a rounding step.

\noindent
\begin{center}
\begin{minipage}{1.0\linewidth}
\begin{algorithm}[H]
\footnotesize
\caption{$\texttt{V$(1,0)$-Cycle (\V) Correction Scheme}$}
\label{alg-mg}
\begin{algorithmic}[1]
\Require A, r, P, $\ell \ge 1$ {\V } levels.
\State $\textcolor{lo}{r \leftarrow} r$\COMMENT{\textcolor{lo}{Round RHS} and Initialize {\V }}
\State \textcolor{lo}{$y \leftarrow Mr$}\COMMENT{\textcolor{lo}{Relax on Current Approximation}}
\If {$\ell > 1$}  \COMMENT{Check for Coarser Grid}
\State \textcolor{lo}{$r_{\textrm{v}} \leftarrow Ay - r$}\COMMENT{\textcolor{lo}{Evaluate {\V } Residual}}
\State \textcolor{lo}{$r_{\ell - 1} \leftarrow P^tr_{\textrm{v}}$}\COMMENT{\textcolor{lo}{Restrict {\V } Residual to Coarse-Grid}}
\State \textcolor{lo}{$d_{\ell - 1} \leftarrow $\V$(A_{\ell - 1}, r_{\ell - 1}, P_{\ell - 1}, \ell - 1)$}\COMMENT{\textcolor{lo}{Compute Correction from Coarser Grids}}
\State \textcolor{lo}{$d \leftarrow Pd_{\ell - 1}$}\COMMENT{\textcolor{lo}{Interpolate Correction to Fine Grid}}
\State \textcolor{lo}{$y \leftarrow y - d$}\COMMENT{\textcolor{lo}{Update Approximate Solution of $Ay = r$}}
\EndIf
\State\Return $y$\COMMENT{Return Approximate Solution of $Ay = r$}
\end{algorithmic}
\end{algorithm}
\end{minipage}
\end{center}
\vspace{1em}

The theory in \cite{MandelMcCormickBank1987} and the references cited therein establish optimal energy convergence in infinite precision of Algorithm~\ref{alg-mg} under fairly general conditions for fully regular elliptic PDEs discretized by standard finite elements. We simply assume this to be the case by supposing that the error propagation matrix $V_j$ for level $j$ is bounded by a constant $\rvstar \in [0, 1)$ for all $j$, that is, $\|V_j\|_{A_j} \le \rvstar,  1 \le j \le \ell .$

One aim of this paper is to verify the theory in \cite{McCormick2020} for multigrid applied to a large class of PDEs, including the model problem introduced below. Accordingly, we have in mind matrices whose condition numbers depend on the mesh size, $h$. (While we do not explicitly exclude coarsening in terms of the degree, $p$, of the discretization, our focus is on coarsening in terms of $h$.) To abstract this dependence, define the {\em pseudo mesh size} by $h_j = \kappa^{-\frac{1}{2\emm}} (A_j), \, 1 \le j \le \ell$, where $\emm$ is a positive integer, and the {\em mesh coarsening factor} by $\theta_j =  \frac{h_{j - 1}}{h_j} , \quad 2 \le j \le \ell .$
In the geometric setting, $2\emm$ correspond to the order of the PDE. Under standard assumptions for finite element discretizations, classical theory shows that the condition number on a given grid is bounded by a constant (depending on the finite element approximation order) times $h_{min}^{-2m}$, where $h_{min}$ is the smallest element size on that grid (see \cite[Sec.~5.2]{Strang2008}). In this case, $h_j$ is therefore bounded by that constant times the grid $j$ mesh size.

To allow for a {\em progressive-precision} V-cycle, where precision is tailored to each grid in the hierarchy, assume now that $\ewed$ varies by letting $\ewed_j$ denote the unit roundoff used on level $j, \, 1 \le j \le \ell$. We use similar notation for other parameters that may now depend on the grid level, but suppress the subscript when the level is understood. Specifically, $\ewed_j$-precision is used on level $j$ to store the data, perform relaxation, transfer residuals to level $j-1$ and corrections to level $j + 1$, and round residuals transferred from level $j + 1$.  Define the {\em precision coarsening factor} by
$\dot{\zeta}_j = \frac{\ewed_{j - 1}}{\ewed_j}, \, 2 \le j \le \ell.$
 We will revisit the actual value of $\dot{\zeta}_j$ in Section~\ref{sec:precision}. To accommodate the use of a geometric series involving the rounding-error effects on each level of the V-cycle, denote the {\em coarsening ratio} by $\cratio  = \min_{1 \le j \le \ell} \{\theta_j \zeta_j^{-\frac{1}{m}}\}$ and assume that $\cratio > 1$. To account for rounding errors in relaxation, suppose that a constant $\alpha_{M_j}$ exists such that computing $\m_j z_j$ for any vector $z_j$ on level $j$ in $\ewed_j$-precision yields
\begin{equation}
\m_j  z + \delta_{M_j}, \qquad \|\delta_{M_j}\| \le \alpha_{M_j} \ewed_j \|z_j\|, \, 1 \le j \le \ell.
\label{mest}
\end{equation}
For example, Richardson iteration with $M_j = \frac{\omega}{\|A_j\|},\, 0<\omega<2$, yields $\alpha_{M_j} < \frac{2}{\|A_j\|}$ (see \cite{McCormick2020}). Assume further that relaxation is monotonically convergent in energy: $\|I_j - M_j A_j\|_A < 1$.
To simplify what follows, assume that $\cc$ is a constant such that
\[\cc \ge (1 + \ewed_j) \max \{ \alpha_{M_j} \|A_j\|, \||A_j|\| \alpha_{M_j} , \||A_j|\| \cdot \|M_j\|\}\}, \, 1 \le j \le \ell.
\]
Only the low precision varies by level in the V-cycle because its finest level is fixed. On the other hand, the full multigrid algorithm introduced below uses progressively finer grids for its inner-loop V-cycles. We therefore introduce variable $\ewe_j$ and $\eweb_j$, $1 \le j \le \ell$, for this purpose, where $\ell$ is now the very finest level used in FMG. Finally, we redefine the following parameters to mean their maxima over all levels:
\begin{align*}
\kappa (P^tP) &= \max_{1 \le j \le \ell} \kappa (P_j^tP_j), \quad \aladot = \max_{1 \le j \le \ell}\frac{m_{A_j}}{1 - m_{A_j} \ewed_j},\quad \alp = \max_{1 \le j \le \ell}\frac{m_{P_j}}{1 - m_{P_j} \ewed_j},
\\
& \alabar = \max_{1 \le j \le \ell}\frac{m_{A_j}}{1 - m_{A_j} \eweb_j}, \quad \textrm{and}\quad \alpnodot = \max_{1 \le j \le \ell}\frac{m_{P_j}}{1 - m_{P_j} \ewe_j}.
\end{align*}

The next theorem confirms that {\V } reduces the error optimally toward the solution of the target matrix equation $Ay = r$ provided that the perturbation $\pv$ of the exact convergence factor satisfies $\pv < 1 - \rvstar$, meaning that coarsening in the grid hierarchy should be fast enough (i.e., large enough $\theta_j$) and progression of the precision should be slow enough (i.e., small enough $\dot{\zeta}_j$) to ensure that $\cratio \gg 1$. More significantly, it requires the finest-level scale parameter to satisfy $\pvard \ll 1$, which in turn means that $\kappa(A) \ll \ewed^{-2}$. Together with Theorem~\ref{theorem2}, we can then conclude that the mixed-precision version of {\ir } with a {\V } as the inner loop converges optimally to the solution of (\ref{axb}) until to the order of the lower limit $\thresh$ is reached.

\begin{thm}{\em \V.  \cite{McCormick2020}}
Define the following quadratic polynomial in $\pvard_j$:
\begin{equation}
\pv(\pvard_j) = \frac{\cratio^\emm}{\cratio^\emm - 1}\left(4\pvard_j + (2 + \bet) \ead_j + \hot_j \right),
\label{rh-v}
\end{equation}
where $\hot_j = 2\pvard^2_j + (4 + \bet)\ead_j \pvard_j + 2 \ead_j \pvard^2_j,
\ead_j = 3 \dot{\zeta}_j\kappa^\frac{1}{2}(P^tP)  \alp \pvard_j,$ and $\bet = 2 + 3 \cc + 2\aladot (1 + \cc)$.
If $\pvard_j$ is small enough that $\pv(\pvard_j) < 1 - \rvstar, \, 1 \le j \le \ell$, then one cycle of the progressive-precision version of Algorithm~\ref{alg-mg} for solving the {\ir } residual equation converges according to $\|y - A^{-1}r\|_A \le \rv \|A^{-1}r\|_A, \, \rv = \rvstar + \pv \, (\pv = \pv(\pvard_\ell))$.
\label{theorem3}
\end{thm}

Full multigrid uses a special cycling scheme that targets the underlying PDE, with the aim of attaining accuracy comparable to how well the finest-grid solution approximates the PDE solution. FMG starts on the coarsest grid and proceeds to the finest, making sure that enough V-cycles are used on each grid along the way to achieve accuracy comparable to that grid's discretization accuracy. In essence, if grid $j - 1$ is solved to within the discretization error $Ch_{j - 1}^q$ for some positive constants $C$ and $q $, then using that result as an initial guess on grid $j$ means that the initial error on grid $j$ is bounded by some small multiple (depending on $\theta_j$) of $Ch_j^q$. This in turn means that only a few V-cycles are needed to obtain discretization accuracy on grid $j$ (i.e., error below $Ch_j^q$). For standard finite elements $q=k-m$ where $2m$ and $k > m$ correspond to the order of the PDE and the order of the finite elements (e.g., polynomials of degree $p=k-1$), respectively \cite[Sec. 2.2]{Strang2008}. 

The full multigrid algorithm based on $\I \ge 1$ inner {\ir } cycles each using one {\V } is given below by the pseudocode {\fmg }. Note that {\fmg}  amounts to three nested loops: outer \fmg, middle \ir, and inner \V. The choice of $\I$ is critical because it must guarantee convergence to within discretization accuracy on each level. The goal of Section~\ref{coarsening} is to determine $\I$ in the presence of rounding errors.

\noindent
\begin{center}
\begin{minipage}{0.97\linewidth}
\begin{algorithm}[H]
\footnotesize
\caption{$\texttt{FMG$(1,0)$-Cycle (\fmg)}$}
\label{alg-fmg}
\begin{algorithmic}[1]
\Require A, b, P, $\I \ge 1$ {\ir } cycles (using one V$(1,0)$ each), $\ell \ge 1$ {\fmg } levels.
\State $x \leftarrow 0$\COMMENT{Initialize {\fmg } }
\If {$\ell > 1$}  \COMMENT{Check for Coarser Grid}
\State \textcolor{med}{$x_{\ell - 1} \leftarrow $\fmg$(A_{\ell - 1}, b_{\ell - 1}, P_{\ell - 1}, \ell - 1, \I)$}\COMMENT{\textcolor{med}{Compute Coarse-Grid Approximation}}
\State \textcolor{med}{$x \leftarrow P x_{\ell - 1}$}\COMMENT{\textcolor{med}{Interpolate Approximation to Fine Grid}}
\EndIf
\State $i \leftarrow 0$\COMMENT{Initialize {\ir } }
\While {$i < \I$}
\State \textcolor{med}{$r \leftarrow$}\textcolor{red}{$ Ax - b$} \COMMENT{\textcolor{red}{Update {\ir } Residual} and \textcolor{med}{Round}}
\State \textcolor{lo}{$y \leftarrow $\V$(A, r, P, \ell)$}\COMMENT{\textcolor{lo}{Compute Correction by \V}}
\State $i \leftarrow i + 1$\COMMENT{Increment {\ir } Cycle Counter}
\State \textcolor{med}{$x \leftarrow x - y$}\COMMENT{\textcolor{med}{Update Approximate Solution of $Ax = b$}}
\EndWhile
\State\Return $x$\COMMENT{Return Approximate Solution of $Ax = b$}
\end{algorithmic}
\end{algorithm}
\end{minipage}
\end{center}
\vspace{1em}

To obtain an abstract sense of discretization accuracy, assume that $b_{j - 1} = P_j^tb_j$, $2 \le j \le \ell$, are also computed exactly. We characterize the relative accuracy of adjacent levels in the grid hierarchy by assuming that $C$ is a positive constant such that the following {\em strong approximation property (SAP)} holds:
\begin{equation}
\|P_j A_{j - 1}^{-1}b_{j - 1} - A_j^{-1}b_j\|_{A_j} \le C h_{j - 1}^q \|A_j^{-1}b_j\|_{A_j}, \quad 2 \le j \le \ell , \quad q = k - m,
\label{C}
\end{equation}
($C$ and $h_j$ may depend on $k$, but we assume that this order is fixed in what follows.)
While (\ref{C}) characterizes the relative error in a coarse-grid solution with respect to the next finer grid, it also suggests the following definition. We say that $x_j$ solves $A_jx_j =b_j$ {\em to the order of discretization error} or simply {\em to discretization accuracy} if
\begin{equation}
\|x_j - A_j^{-1}b_j\|_{A_j} \le C h_j^q  \|A_j^{-1}b_j\|_{A_j} , \quad 1 \le j \le \ell.
\label{disc}
\end{equation}
We assume that this level of approximation is achieved on the coarsest level $j=1$ by just a few relaxation sweeps starting with a zero initial guess.

\begin{thm}{\em \fmg. \cite{McCormick2020}}
Assume that $\rh_v + \pir < 1$ and that $\thresh$ is small enough and $\I$ is large enough that the following  holds on all levels $j \in \{1, 2,\dots, \ell\}$:
\begin{equation}
(\rh_v + \pir)^\I \left((\sqrt{2} + \ea) \theta^q C h^q +  \ea\right) + \frac{\thresh}{1 -(\rh_v + \pir)} \le C h^q ,
\label{mess}
\end{equation}
where $h = h_j$, $\theta = \theta_j$, $\ea = \ea_{j} = 3 \dot{\zeta}_j\kappa^\frac{1}{2}(P^tP)  \alpnodot \pvar_j$, and (with subscript $j$ understood) the parameters $\rh_v$, $\pir$, and $\thresh$ are given by (\ref{rh-v}), (\ref{ec}),  and (\ref{chi}), respectively. Then Algorithm~\ref{alg-fmg} solves (\ref{axb}) to the order of discretization error on each level.
\label{theorem4}
\end{thm}

\section{Effects of Quantization \& Galerkin Construction on {\V } \& {\ir}}
\label{galerkin}

The components $A_j, P_j$, and $b_j$ have so far been assumed to be exact for all $j$. In this and the next section, we extend the theory from \cite{McCormick2020} to include \emph{quantization errors} incurred from simply storing the components in finite precision. This is in addition to the \emph{algebraic errors} accumulated during computations and already accounted for in the existing theory.

Dropping subscript $j$, assume that the system matrices are stored in symmetric form and that $\check{A} = A + \Delta$ and $\check{b} = b + \delta$ result from simply rounding the exact $A$ and $b$, respectively, to some $\eweq$-precision. The actual value of $\eweq$ will be determined later. We first obtain the general result that $(A + \Delta)^{-1} b \approx A^{-1} b$ and $\kappa(A + \Delta) \approx \kappa(A)$ to the extent that $\kap \eweq < 1$.

\begin{thm}{\em $A$ and $b$ Quantization Errors.}
If $\kap \eweq < 1$, then $A + \Delta$ is SPD and $(A + \Delta)^{-1}(b + \delta)$ approximates $A^{-1}b$ with relative error bounded according to
\begin{equation}
\frac{\|(A + \Delta)^{-1}(b + \delta) - A^{-1}b\|_A}{\|A^{-1}b\|_A} \le \phi\eweq, \quad \phi = \frac{\kap + \kappa^\frac{1}{2}(A)}{1 - \kap \eweq}.
\label{relerr}
\end{equation}
\label{abround}
\end{thm}
\begin{proof}
The relative error in each entry of $A + \Delta$ and $b + \delta$ is bounded by $\eweq$, which immediately yields the relative error bound
\begin{equation}
\|\delta\| \le \|b\|\eweq.
\label{b}
\end{equation}
Using $\cdot$ to emphasize multiplication, a bound for $A + \Delta$ follows by noting that $y^ty = |y|^t|y|$ and
$y^t\Delta \cdot y = |y^t\Delta \cdot y| \le |y|^t|\Delta |\cdot|y| \le |y|^t|A|\cdot|y|\eweq$
for any $y \in \Re^n$:
\begin{equation}
\|\Delta \| \le \||A|\| \eweq.
\label{A}
\end{equation}
By  (\ref{A}) and noting that $A + \Delta  = A^\frac{1}{2}\left(I + E\right)A^\frac{1}{2}, E = A^{-\frac{1}{2}}\Delta A^{-\frac{1}{2}}$, we have that
\begin{equation}
\|E\| \le \|\Delta \| \cdot \|A^{-1}\| \le  \||A|\|\cdot\|A^{-1}\|\eweq = \kap \eweq < 1,
\label{E}
\end{equation}
which proves that $I + E$ and, hence, $A + \Delta $ are positive definite. Note also that
\begin{equation}
\|\left(I + E\right)^{-1}\| \le \frac{1}{1 - \|E\|} \le \frac{1}{1 - \kap \eweq}.
\label{inv}
\end{equation}

We next obtain the following expression for the perturbation of $A^{-1}b$:
\begin{align}
(A + \Delta)^{-1}(b + \delta) - A^{-1}b
&=A^{-\frac{1}{2}}\left[\left(I + E\right)^{-1} - I\right]A^{-\frac{1}{2}}b + A^{-\frac{1}{2}}\left(I + E\right)^{-1}A^{-\frac{1}{2}}\delta \nonumber \\
&= A^{-\frac{1}{2}}\left(I + E\right)^{-1}\left(- E A^\frac{1}{2} A^{-1}b + A^{-\frac{1}{2}}\delta\right).
\label{del}
\end{align}
Finally, noting that $\|b\| = \|A^{\frac{1}{2}}A^{-\frac{1}{2}}b\| \le \|A^{\frac{1}{2}}\|\cdot\|A^{-\frac{1}{2}}b\|$, $\|w\|_A = \|A^\frac{1}{2}w\|$ for any vector $w$, and $\kappa(A)^{\frac{1}{2}} = \|A^{-\frac{1}{2}}\|\cdot \|A^\frac{1}{2}\|$, then the theorem is proved as follows:
\begin{align*}
\|(A +& \Delta)^{-1}(b + \delta) - A^{-1}b\|_A \\
&\le \|\left(I + E\right)^{-1}\|\left(\|E\|\cdot\|A^{-1}b\|_A + \|A^{-\frac{1}{2}}\delta\|\right) &\textrm{by (\ref{del})} \\
&\le \frac{1}{1 - \kap \eweq}\left(\|E\|\cdot\|A^{-1}b\|_A + \|A^{-\frac{1}{2}}\|\cdot\|\delta\|\right) &\textrm{by (\ref{inv})} \\
&\le \phi\eweq\|A^{-1}b\|_A &\textrm{by (\ref{b}) and (\ref{E})} .
\end{align*}
\end{proof}

When the multigrid components are extracted directly from the discretization, quantization in $\ewed_j$-precision incurs a relative $\kapj \ewed_j$ error in these components. However, our framework also applies to inherently algebraic problems, with the coarse-grid matrices in {\V } possibly constructed based on the Galerkin condition. For simplicity in illustrating rounding effects for this case, we consider a single level $j - 1$ only, assuming that $A_j$ and $P_j$ are exact and the Galerkin condition is computed in $\ewed_{j - 1}$-precision in the order given by $P_j^t\left(A_jP_j\right)$. Our next theorem shows that the resulting rounding errors are also $\bigO(\kapj \ewed_j)$.

\begin{thm}{\em Galerkin Rounding Errors.}
Fix $j \in \{1, 2,\dots, \ell-1\}$ and assume that $A_j$ and $P_j$ are exact. Then the coarse-grid matrix $A_{j - 1} = P_j^tA_jP_j + \Delta$ computed from the Galerkin condition in $\ewed_{j - 1}$-precision satisfies the relative error bound
\begin{equation}
\|\Delta\| \le (1 + 2\ewed_{j - 1})\ewed_{j - 1} \|P_j^t|A_j|P_j\| \le \kapj (1 + 2\ewed_{j - 1})\ewed_{j - 1} \|P_j^tA_jP_j\|.
\label{gerr}
\end{equation}
\label{galerror}
\end{thm}
\begin{proof}
Write the computed $A_jP_j$ as $A_jP_j + \Delta_1$,  $|\Delta_1| \le |A_j| P_j \ewed_{j - 1}$. Then we can write $A_{j - 1} = P_j^t\left(A_jP_j + \Delta_1\right) + \Delta_2$,  $|\Delta_2| \le  P_j^t |A_jP_j + \Delta_1| \ewed_{j - 1}$. We thus have that
\[
|\Delta| = |P_j^t\Delta_1 + \Delta_2| \le \left( 1 + 2\ewed_{j - 1} \right)\ewed_{j - 1} P_j^t|A_j|P_j \le \kapj \left( 1 + 2\ewed_{j - 1} \right)\ewed_{j - 1} P_j^tA_jP_j ,
\]
which follows from $|A_j| \le \||A_j|\| I \le \kapj A_j$. Taking norms proves the theorem.
\end{proof}
Theorem~\ref{abround} suggests that $\kapj\ewed_j \ll 1$ is needed to ensure good V-cycles performance. After all, if $A_j$ is indefinite, then just computing the residual could expand the error associated with the negative spectrum. But this is not really a concern for {\V}: quantization has negligible effect in $\ewed_j$-precision on {\V} because this error expansion is small compared to other rounding errors, as our next theorem shows. Note that a similar result holds when all $A_j$ are computed via the Galerkin condition, where (\ref{gerr}) would be used recursively to account for errors accumulated over all levels. Note also that quantization of $M_j$ would have a truly negligible effect on the performance of {\V } because preconditioners only need to be crude approximations to the inverse (e.g., relaxation parameters are typically allowed to be anywhere in the interval $(0,2)$).

\begin{thm}{\em {\V } Quantization Errors.}
Let $A_j$ quantized in $\ewed_j$-precision be denoted by  $A_j + \Delta_j, \, |\Delta_j| \le |A_j| \ewed_j$, $1 \le j \le \ell$. Then Theorem~\ref{theorem3} holds with $\bet$ replaced by the slightly larger $\bet = 2 + 3 \cc + 2(\aladot + 1)(1 + \ewed_1)(1 + \cc)$.
\label{quant-v}
\end{thm}
\begin{proof}
We treat each level individually because the exact $A_j$ is rounded directly, without error accumulation. Dropping subscript $j$, since $A$ is only used in {\V } for computing the residual just before coarsening, all we need do is establish a quantized version of bound (37) in the proof of Theorem 2 of \cite{McCormick2020}, which is of the form
\[
r^{(\frac{1}{2})} =  Ay - r +  \delta_1, \quad
|\delta_1| \le \aladot \ewed \left(|r| + |A| \cdot |y|\right).
\]
Substituting in quantized $A$ yields $r =  (A + \Delta)y - r +  \delta_2$, where

\[
|\delta_2| \le \aladot \ewed \left(|r| + |A + \Delta| \cdot |y|\right) \le \aladot (1 + \ewed)\ewed \left(|r| + |A| \cdot |y|\right).
\]
Thus, $r^{(\frac{1}{2})} =  Ay - r + \delta_3$, where
\[
|\delta_3| = |\Delta \cdot y + \delta_2| \le (\aladot + 1) (1 + \ewed_1) \ewed \left(|r| + |A| \cdot |y|\right),
\]
where we replaced $\ewed_j$ by $\ewed_1 \ge \ewed_j$ to ensure that the change in $\bet$ is just $\aladot$ replaced with the slightly larger {\em constant} $(\aladot + 1) (1 + \ewed_1)$. This completes the proof.
\end{proof}
\begin{rem}{\em Sensitivity of {\V } to Quantization.}
Theorem~\ref{quant-v} confirms that {\V } is insulated from the indefiniteness that quantization may create. This insensitivity comes from the fact that V-cycles are basically just a hierarchy of simple relaxation steps that have little effect on the near-kernel error components that indefiniteness may alter. On coarse enough grids, relaxation may begin to significantly affect the near-kernel components, but this is just where the system matrices retain positive definiteness (because the condition numbers are small). Other basic relaxation methods may also be insensitivity to quantization, but they tend not to be very efficient solvers for PDEs. On the other hand, while direct solvers can be applied to modest-size discrete PDEs, their reliance on positive definiteness to control the error makes them very sensitive to quantization.
\end{rem}
\begin{thm}{\em {\ir } Quantization Errors.}
Let $A$ and $b$ quantized in $\eweq$-precision be denoted by $A + \Delta, \, |\Delta| \le |A| \eweq,$ and $b + \delta, \, |\delta| \le |b| \eweq,$ respectively. Then Theorem~\ref{theorem2} holds with $(1 + \ewe)\alabar \pvarb$ replaced by $(1 + \ewe)\left(\pvarq + (1 + \eweq)\alabar\pvarb\right)$ in the expressions for $\pir$ and $\thresh$ in (\ref{chi}) , where $\pvarq = \kappa(A) \eweq$.
\label{quant-ir}
\end{thm}
\begin{proof}
The proof is analogous to that of Theorem~\ref{quant-v}, but with three terms quantized in bound (13) in the proof of Theorem 1 of \cite{McCormick2020}, which for exact $A$ and $b$ reads
\[
r = Ax - b + \delta_1,
\quad |\delta_1| \le \ewe |Ax - b| + (1 + \ewe)\alabar {\eweb} \left(|b| + |A| \cdot |x|\right).
\]
Substituting in quantized $A$ and $b$ thus yields $r = (A + \Delta)x - (b + \delta) + \delta_1$, where
\begin{align*}
|\delta_1| & \le \ewe |(A + \Delta)x - (b + \delta)| + (1 + \ewe)\alabar {\eweb} \left(|b + \delta| + |A + \Delta| \cdot |x|\right) \\
& \le \ewe |Ax - b| + \ewe (|\delta| + |\Delta|\cdot|x|) + (1 + \ewe)\alabar {\eweb} \left(|b| + |A| \cdot |x| + |\delta| + |\Delta|\cdot|x|\right) \\
& \le \ewe |Ax - b| + \left(\ewe \eweq + (1 + \ewe)(1 + \eweq)\alabar {\eweb}\right) \left(|b| + |A| \cdot |x|\right).
\end{align*}
Thus, $r = (A + \Delta)x - (b + \delta) + \delta_1 = Ax - b + \delta_2$, and the theorem follows because
\begin{align*}
|\delta_2| & = |\Delta \cdot x - \delta + \delta_1| \\
& \le \ewe |Ax - b| + |\delta| + |\Delta| \cdot |x| + \left(\ewe \eweq + (1 + \ewe)(1 + \eweq)\alabar {\eweb}\right) \left(|b| + |A| \cdot |x|\right) \\
& \le \ewe |Ax - b| + (1 + \ewe)\left(\eweq + (1 + \eweq)\alabar\eweb\right)\left(|b| + |A| \cdot |x|\right).
\end{align*}
\end{proof}

\section{Effect of Input Quantization on \fmg}
\label{coarsening}
{\fmg } is more sensitive to quantization because it relies directly in step 4 of Algorithm~\ref{alg-fmg} on the SAP (\ref{C}). (We assume from now on that (\ref{C}) holds when $A, b,$ and $P$ are exact on all levels.) Here we analyze the effects on {\fmg} of rounding $A_{j - 1}, P_j$, and $b_{j - 1}$ for a fixed $j \in \{2,3,\dots,\ell\}$ to a given {\em quantization precision} $\eweq$. To clarify where this rounding occurs, note that each recursive call to {\fmg } means that $j-1$ serves as the finest level for the inner {\V } calls on {\fmg } level $j - 1$. So $A_{j - 1}$ and $b_{j - 1}$ rounded to precision $\eweq_{j - 1} \le \ewe_{j - 1}$ in {\fmg } means that these rounded quantities are passed into the recursive call to {\fmg } from level $j$ to the coarser level. Note that the resulting $A_{j - 1}$ is further rounded to $\ewed_{j - 1}$-precision in the inner call to \V. Similarly, rounding $P_j$ to $\eweq_j$-precision in {\fmg } means that this occurs in step 4 when the full approximation is interpolated from the current finest grid $j - 1$ to the new finest grid $j$. All other multigrid components are processed in $\ewed$-precision within the inner {\V } solver. Our final theorem extends Theorem~4 in \cite{McCormick2020} to account for these quantization  errors, at the cost of increased complexity. Aligned with our ultimate goal of balancing errors, the aim here is for the solver and rounding errors to each be smaller than $Ch^q$, as opposed to bounding their sum as in (\ref{mess}). The key to this extension is to establish a SAP that accounts for quantization. Specifically, with $A_j^{-1}b_j + \delta_j$ denoting the exact solution of $A_j x_j = b_j$ when $A_j$ and $b_j$ have been quantized, then the extended SAP asserts existence of a constant $\check{C}$ such that
\begin{equation}
\|P_j A_{j - 1}^{-1}b_{j - 1}+\delta_{j - 1} - (A_j^{-1}b_j + \delta_j)\|_{A_j}\le \check{C} h_{j - 1}^q \|A_j^{-1}b_j + \delta_j\|_{A_j}, \quad 2 \le j \le \ell.
\label{eC}
\end{equation}
\begin{thm}{\em {\fmg } Quantization Errors.} The extended SAP (\ref{eC}) holds with
\[
\check{C} = \max_{2 \le j \le \ell}\{\left(C + \kappa^\frac{q}{2m}(A_{j - 1}) \left(\phi_{j - 1} \eweq_{j - 1} + \phi_j \eweq_j + \kappa_j \eweq_j \left(1 + \phi_{j - 1} \eweq_{j - 1} \right)\right)\right)(1 + \phi_j \eweq_j)\},
\label{sap-C}
\]
where  $\phi_j = \frac{\kapj +  \kappa^\frac{1}{2}(A_j)}{1 - \kapj\eweq_j}$ and $\kappa_j = \kappa^\frac{1}{2}(A_j) \underline{\kappa}^\frac{1}{2}(A_{j - 1}), \, 2 \le j \le \ell.$ Moreover, {\fmg } approximates the solution $A^{-1}b + \delta$ of the quantized version of (\ref{axb}) to the level of discretization accuracy provided $\rh_v + \pir < 1$ and the following hold on every level:
\begin{equation}
\frac{\thresh}{1 -(\rh_v + \pir)} < C h^q \quad\textrm{  and  }\quad
(\rh_v + \pir)^\I  \left(\theta^q C_c h^q +\mu_c\right) \le C h^q,  
\label{Nquant}
\end{equation}
where the constants $C_c = \max_j\{(1 + \kappa^\frac{1}{2}(A_j)\eweq_j)(1 + \phi_{j - 1} \eweq_{j - 1})(1 + \phi_j \eweq_j)C + \check{C}\} $ and $\mu_c = \max_j{\eweq_j(1 + \eweq_j)\kappa^\frac{1}{2}(A_j)(1 + C h_{j - 1}^q)(1 + \phi_{j - 1} \eweq_{j - 1})(1 + \phi_j \eweq_j)}$, and subscript $j$ is understood for the other terms in (\ref{Nquant}).
\end{thm}
\begin{proof}
Dropping subscript $j$ and replacing subscript $j - 1$ by $c$, let $P + \Delta_P$ denote quantized $P$, where $|\Delta_P \cdot z| \le |\Delta_P| \cdot |z| \le \eweq P \cdot |z|$ for any coarse-grid $z$. This proof uses the bounds $\|\delta_c\|_{A_c} \le \phi_c \eweq_c \|A_c^{-1}b_c\|_{A_c}$, $\|\delta\|_{A} \le \phi \eweq \|A^{-1}b\|_{A}$, and $\|A_c^{-1}b_c + \delta_c\|_{A_c} \le (1 + \phi_c \eweq_c)(1 + \phi \eweq) \|A^{-1}b + \delta\|_{A}$ that are implied by (\ref{relerr}). The proof assumes familiarity with the logic as well as some estimates used in \cite{McCormick2020}, including $\||z|\|_{A_c} \le \|A_c^\frac{1}{2}\|\cdot\|z\| \le \kapjm^\frac{1}{2} \|z\|_{A_c}$ and $\|A_c^{-1}b_c\|_{A_c} \le \|A^{-1}b\|_A$.

To establish the extended SAP (\ref{eC}), first note that
\begin{align*}
\|\Delta_P\left(A_c^{-1}b_c + \delta_c\right)\|_A &\le \|A^{\frac{1}{2}}\|\eweq \|P\cdot |A_c^{-1}b_c + \delta_c | \| \\
&\le \kappa^\frac{1}{2}(A) \eweq \||A_c^{-1}b_c + \delta_c | \|_{A_c} \\
&\le \kappa \eweq \left(1 + \phi_c \eweq_c \right) \|A_c^{-1}b_c \|_{A_c} .
\end{align*}
Then $h_c = \kappa^{-\frac{1}{2\emm}} (A_c)$, (48) in \cite{McCormick2020}, and the original SAP (\ref{C}) establish (\ref{eC}):
\begin{align*}
\|\left(P + \Delta_P\right)& \left(A_c^{-1}b_c + \delta_c\right) - \left(A^{-1}b + \delta\right)\|_A \\
&\le \|P A_c^{-1}b_c - A^{-1}b\|_A + \|\delta\|_A + \|P \delta_c\|_A + \|\Delta_P\left(A_c^{-1}b_c + \delta_c\right)\|_A \\
&\le (Ch_c^q + \phi \eweq)\|A^{-1}b\|_A + \left(\phi_c \eweq_c + \kappa \eweq \left(1 + \phi_c \eweq_c \right)\right)  \|A_c^{-1}b_c \|_{A_c} \\
&\le \left( C + \kappa^\frac{q}{2m}(A_c) \left(\phi_c \eweq_c + \phi \eweq + \kappa \eweq \left(1 + \phi_c \eweq_c \right) \right) \right) h_c^q \|A^{-1}b\|_A \\
&\le \check{C} h_c^q \|A^{-1}b + \delta\|_A .
\end{align*}
For FMG convergence, assume for induction purposes that the coarse-grid result, $x_c$, has properly converged:  $\|x_c - (A_c^{-1}b_c + \delta_c)\|_{A_c} \le C h_c^q \|A_c^{-1}b_c + \delta_c\|_{A_c}$. Then
\begin{align*}
\|\left(P + \Delta_P\right)& \left(x_c - (A_c^{-1}b_c + \delta_c)\right)\|_A \\
&\le \|P\left(x_c - (A_c^{-1}b_c + \delta_c)\right)\|_A + \|\Delta_P \left(x_c - (A_c^{-1}b_c + \delta_c)\right)\|_A \\
&\le (1 + \kappa^\frac{1}{2}(A)\eweq)C h_c^q \|A_c^{-1}b_c + \delta_c\|_{A_c} ,
\end{align*}
which with (\ref{eC})  implies that
\begin{align}
&\|\left(P + \Delta_P\right) x_c - (A^{-1}b + \delta)\|_A \nonumber \\
&\le \|\left(P + \Delta_P\right) \left(x_c - (A_c^{-1}b_c + \delta_c)\right)\|_A + \||\left(P + \Delta_P\right) (A_c^{-1}b_c + \delta_c) - (A^{-1}b + \delta)\|_A \nonumber\\
&\le \left((1 + \kappa^\frac{1}{2}(A)\eweq)(1 + \phi_c \eweq_c)(1 + \phi \eweq)C + \check{C}\right) h_c^q \|A^{-1}b +\delta\|_{A}.
\label{two}
\end{align}
Denote $\left(P + \Delta_P\right)x_c$ computed in $\eweq$-precision by $\left(P + \Delta_P\right)x_c + \delta_x$, where $|\delta_x| \le \eweq \left(P + \Delta_P\right)|x_c| \le \eweq (1 + \eweq)P|x_c|$. But $\|\delta_x\|_{A_c} \le  \eweq(1 + \eweq)\kappa^\frac{1}{2}(A)\|x_c\|_{A_c}$ and
\begin{align*}
\|x_c\|_{A_c} &\le \|x_c\|_{A_c} + \|x_c - (A_c^{-1}b_c + \delta_c)\|_{A_c} \\
&\le (1 + C h_c^q) \|A_c^{-1}b_c + \delta_c\|_{A_c} \\
&\le (1 + C h_c^q)(1 + \phi_c \eweq_c)(1 + \phi \eweq) \|A^{-1}b + \delta\|_{A} .
\end{align*}
Thus, $|\delta_x\|_{A_c} \le \eweq(1 + \eweq)\kappa^\frac{1}{2}(A)(1 + C h_c^q)(1 + \phi_c \eweq_c)(1 + \phi \eweq) \|A^{-1}b + \delta\|_{A}$, which with (\ref{two}) confirms that $\|x - (A^{-1}b + \delta)\|_A \le C_c \|A^{-1}b + \delta\|_{A}$, where $x = \left(P + \Delta_P\right) x_c + \delta_x$ is the initial iterate for the V-cycles on the fine grid. As in the proof of Theorem~4 in \cite{McCormick2020}, we then invoke Theorems~\ref{theorem2} and \ref{theorem3} above to prove the theorem.
\end{proof}
The condition on $N$ in (\ref{Nquant}) can be substantially simplified by noting that $\kapjm \le \kappa(A_j)$ and $\kappa^\frac{1}{2}(A_j) \ll \kappa(A_j)$, by assuming that $\kapj \eweq_j \approx \kappa(A_j) \eweq_j \ll 1$ and $\kappa^\frac{q + 2m}{2m}(A_j)\eweq_j \lesssim C$, and by deleting negligible terms. We therefore conclude that $1 + \kapj \eweq \lesssim 1 + \kappa^\frac{q}{2m}(A_j)C \approx 1$, $\phi_j \approx \kapj$, and $\kappa^\frac{q}{2m}(A_j) \phi_j \approx C$, and similarly for other analogous terms. We then have that $\check{C} \approx 4C$, that $C_c \approx C + \check{C} \approx 5C$, and that $\mu_c \approx \kappa^\frac{1}{2}(A_j) \eweq_j \ll C$, leading to the simplified condition $5 {\rvstar}^\I \theta^q C h^q \le C h^q $, i.e., $5 {\rvstar}^\I \theta^q \lesssim 1$. $\I$ thus requires a relatively modest increase from $\frac{\log_2(\sqrt{2}) + q\log_2(\theta)}{| \log_2(\rvstar)|}$ in Remark~3 of \cite{McCormick2020} to the following estimate that accounts for quantization:
\begin{equation}
\I \approx \frac{\log_2(5) + q\log_2(\theta)}{| \log_2(\rvstar)|}.
\label{eN}
\end{equation}

\section{Precision Requirements}
\label{sec:precision}

Up to this point, we have referred to $\ewe$, $\eweb$, and $\ewed$ as standard, high, and low precision, respectively, without specifying how to select these precisions. Additionally, we have introduced $\eweq$ for the quantization precision. In this section, we show how the theoretical estimates can guide the selection of all of these precision levels. While the focus is on FMG, most of the tools discussed here also apply to V-cycles.

Our estimates involve discretization, quantization, rounding, and iteration errors. Ideally, all errors should be comparable so that computation is not wasted on reducing one only to have the end result contaminated by the others. As shown in (\ref{errdecomp}), the total error is bounded by the sum of the four types of errors, so our guiding principle is to assume that the \emph{bounds} for each of these errors should be comparable\footnote{If the cost of reducing different types of errors vary widely, then one could conceivably include weights when allocating the error-budget, but for simplicity we will not include that here.}. This goal leads to overestimates of the total error and the individual precision levels, but such is the nature of an a priori theoretical analysis. Also, while (\ref{errdecomp}) provides a decomposition of the absolute errors, we are able to focus instead on the relative errors by assuming that the total error is small enough that $\|\tilde{u}_h^{(\infty)}\|_{\mathcal{L}} \approx \|\check{u}_h\|_{\mathcal{L}} \approx \|u_h\|_{\mathcal{L}} \approx \|u\|_{\mathcal{L}}$.

We begin by recalling the basic requirement that $\eweb \leq \ewe \leq \ewed$. Additionally, we must have $\eweb \leq \eweq$ since $\eweb$ is the highest precision used for any computation. If $\eweb > \eweq$, then any computation would effectively include a rounding operation to at least $\eweb$-precision, which makes the choice of higher precision for $\eweq$ pointless. On the other hand, if $\eweq$ is strictly greater than $\eweb$, $\ewe$, and/or $\ewed$, then operations can still be performed in the specified precision by extending all $\eweq$-numbers with trailing zeros. 


\newcommand{\offset}{0.4}

\begin{figure}
\centering
\includegraphics{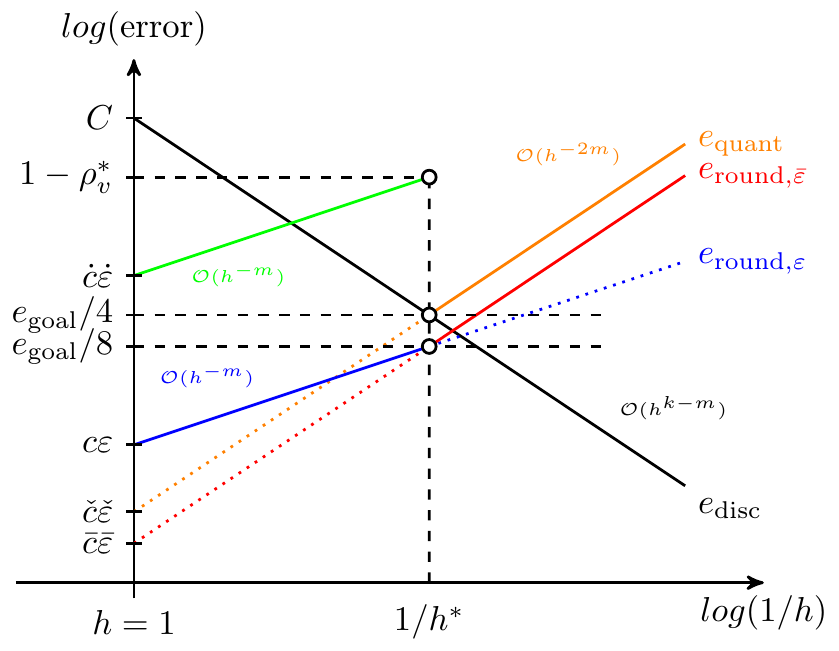}
\caption{Error balance diagram. The color coding corresponds to that used in our pseudo code.}
\label{fig:errbalance}
\end{figure}

To determine the required precision levels, we first illustrate in Figure~\ref{fig:errbalance} the different types of errors along with an assumed desired error-level, $e_{\textrm{goal}}$. Given that there are four contributions to the total error, each must be on the order of $\tfrac{1}{4}e_{\textrm{goal}}$. This level is shown as a one of the dashed horizontal lines in Figure~\ref{fig:errbalance}. The energy norm of the discretization error for a standard finite element discretization is given by $e_{\textrm{disc}} = C h^{q}$, where $q=k-m$ \cite[Sec.~2.2]{Strang2008}. The intersection between this discretization error line and the $\tfrac{1}{4}e_{\textrm{goal}}$ line defines the mesh size, $h^*$, required to obtain the desired accuracy. It may be necessary to round $h^*$ down to the nearest available mesh size. To be in balance, all errors must then be comparable \emph{at this grid resolution}.

The estimates for the quantization and rounding errors that we need in the following involve the matrix condition number, $\kappa$,
which is bounded according to $\kappa \leq c_\kappa h^{-2m}$ \cite[Theorem~5.1]{Strang2008}, where $c_\kappa$ is a constant. In practice, $\kappa$ depends on $k$, and so then must $c_\kappa$. 

Theorem~\ref{abround} effectively states that $e_{\textrm{quant}} \le \phi\eweq$. Assuming that $\kap\eweq \ll 1$ and $\kap \approx \kappa$, it follows that $\phi \approx \kappa \leq c_\kappa h^{-2m}$ and therefore that $e_{\textrm{quant}} \lesssim \check{c}\eweq h^{-2m}$ for $\check{c} = c_\kappa$. This bound for $e_{\textrm{quant}}$ is shown as the orange line in Figure~\ref{fig:errbalance}. The constant $\check{c}$ does not depend on $\eweq$, so assuming that we know $\check{c}$ it follows that we can choose $\eweq$ to obtain the desired quantization error at $h=h^*$. 

A bound for $e_{\textrm{round}}$ is provided by the last term in (\ref{Kbnd}), which by (\ref{chi}) is the sum of terms proportional to $\pvar = \kappa^{\frac{1}{2}}\ewe$ and to $\pvarb = \kappa\eweb$. Assuming that $\pvar \ll 1$, then $e_{\textrm{round},\ewe} \lesssim \pvar \leq c\ewe h^{-m}$ for $c = \sqrt{c_\kappa}$. Similarly, we get $e_{\textrm{round},\eweb} \lesssim \bar{c} \eweb h^{-2m}$ for some constant $\bar{c}$. An estimate of $\bar{c}$ follows from (\ref{chi}) if we again assume that $\kap \approx \kappa$. In this case, $1 < \gamma < 2$ and $1 < (1+\rho)(1+\ewe) < 2$, so $\alabar c_\kappa \lesssim \bar{c}\lesssim 4 \alabar c_\kappa$. 

The two bounds for $e_{\textrm{round},\ewe}$ and $e_{\textrm{round},\eweb}$ are shown in Figure~\ref{fig:errbalance} as blue and red lines, respectively. In mixed precision, $\ewe > \eweb$, which means that $e_{\textrm{round},\ewe}$ dominates the rounding error for small values of $\kappa$, while $e_{\textrm{round},\eweb}$ dominates for sufficiently large values of $\kappa$. To ensure that all errors are balanced, we choose $\ewe$ and $\eweb$ such that $e_{\textrm{round},\ewe} = e_{\textrm{round},\eweb} = \tfrac{1}{8}e_{\textrm{goal}}$ for $h=h^*$, meaning that $e_{\textrm{round}} = \tfrac{1}{4}e_{\textrm{goal}}$. Once again, assuming that we know $c$ and $\bar{c}$, we can then determine $\ewe$ and $\eweb$. While Figure~\ref{fig:errbalance} shows that $\bar{c}\eweb \leq \check{c}\eweq$, it is important to note that this does not necessarily imply that $\eweb < \eweq$. In fact, since $2e_{\textrm{round},\eweb} = e_{\textrm{quant}}$ at $h=h^*$, it follows that if we use the rough estimates for $\bar{c}$ and $\check{c}$, then we expect that $\eweq = 8 \alabar \eweb$, where $\alabar$ can be relatively large for high-order discretizations. 

The choice of $\ewed$ differs from the other precision levels in that we do not need to achieve $\tfrac{1}{4}e_{\textrm{goal}}$ accuracy in $\ewed$-precision, but instead simply need the V-cycle to be convergent in $\ewed$-precision. This means that we must have $\rv = \rvstar + \pv < 1$ at $h=h^*$. For sufficiently small $\pvard$, it follows from Theorem~\ref{theorem3} that $\pv = \bigO(\pvard)$, where $\pvard = \kappa^\frac{1}{2}\ewed$, so that $\pv \approx \dot{c}\ewed h^{-m}$ for some constant $\dot{c}$. This leads to the green line shown in Figure~\ref{fig:errbalance}. In practice, we have not found a reliable way to determine $\dot{c}$ accurately, but simply choosing $\ewed$ such that $\pvard \ll 1$ appears to work well. Concretely, we choose $\ewed = 0.1 / \kappa^\frac{1}{2}$. 


Balancing $e_{\textrm{iter}}$ with the other three errors amounts to restricting $\I$ rather than any precisions. Our goal is to allow discretization accuracy to include a balanced quantization error, so (\ref{eN}) provides the choice for $N$ that we need and it ensures that the {\em four} errors bounding $e_\textrm{total}$ in (\ref{errdecomp}) are approximately balanced. 

We have thus far assumed a single target accuracy. To extend the estimates to progressive precision for the V-cycle, we can choose $h^*$ (instead of $e_{\textrm{goal}}$) as the independent parameter and then repeat the exercise for every level in the multigrid hierarchy. This leads to $\delta_{\rho_{v},j} \approx \dot{c}\ewed_j h_j^{-m} = 1-\rvstar$ for $1 \leq j \leq l$. It then follows easily by considering $\delta_{\rho_{v},j-1}/\delta_{\rho_{v},j}$ that $\dot{\zeta}_j = \ewed_{j-1} / \ewed_j = (h_{j-1}/h_j)^m = \theta_j^m$. Thus, the precision coarsening factor for $\ewed$ is directly related to the mesh coarsening factor and, given $\ewed_j$ for any level, it is straightforward to compute $\ewed_j$ for the other levels. 

For FMG, the values of $\ewe$, $\eweb$, and $\eweq$ are also level dependent. Given $h_j$, the value for $\eweq_j$ follows easily by equating the bounds for $e_{\textrm{quant}}$ and $e_{\textrm{disc}}$, which yields $\eweq_j \leq (C/\check{c}) h_j^{k+m}$. Similarly, we can equate $\tfrac{1}{2}e_{\textrm{disc}}$ with $e_{\textrm{round},\ewe}$ and $e_{\textrm{round},\eweb}$ to get $\ewe_j \leq \tfrac{1}{2}(C/c)h_j^k$ and $\eweb_j \leq \tfrac{1}{2}(C/\bar{c}) h_j^{k+m}$. If $\theta = 2$, as is typical for geometric multigrid, this means that the sizes of the mantissas needed for $\ewed$ and $\ewe$ grow with $m$ and $k$ bits per level, respectively, while the growth for both $\eweb$ and $\eweq$ is $k+m$ bits per level. Since $m$ is typically one or two, this means that the precision required for $\ewed$ grows quite slowly while that for $\eweb$ and $\eweq$ can grow rather quickly for high-order discretizations. 

To estimate the absolute precisions required for a given level, $C$, $\dot{c}$, $c$, $\bar{c}$, and $\check{c}$ must all be determined. We do this empirically in Section~\ref{sec:results} for our model problem.

\section{Model Problem}
\label{sec:modelproblem}

To illustrate our rounding-error estimates, we consider the 1D biharmonic equation given by the following fourth-order ordinary differential equation (ODE) on $\Omega = (0,1)$ with homogeneous Dirichlet conditions on the boundary $\partial \Omega = \{ 0, 1 \}$:
\[
\quad \left\{ \hspace{5pt}
\parbox{5.00in}{
  \noindent Given $f \in L^2(\Omega)$, find $\sol \in C^4(\Omega)$ such that
    \begin{equation*}
      \begin{array}{rl}
       \sol'''' &= f \hspace{5pt} \text{in } \Omega\\
        \sol = \sol' &= 0 \hspace{5pt} \text{on } \partial \Omega.
      \end{array}
    \end{equation*}
  }
\right.
\]
This fourth-order model problem is useful because it leads to very ill-conditioned matrices with severe sensitivity to rounding errors. Although we do not present the results in detail here, we have also studied the 2D biharmonic  and found  the results to be qualitatively similar but computationally more expensive to obtain. 

To discretize the ODE, we apply a standard Bubnov-Galerkin finite element method to its weak form based on the same trial and test spaces given by
\begin{equation*}
\fespace = \left\{ \trialfun \colon \Omega \rightarrow \Re \Big| \trialfun \in H^2(\Omega), \left. \trialfun \right|_{\partial \Omega} = 0, \text{ and} \left. \trialfun' \right|_{\partial \Omega} = 0 \right\}.
\end{equation*}
The variational form then arises via the $L^2$-projection of $\sol'''' = f$ onto an arbitrary test function $\testfun \in \fespace$ followed by two applications of Green's first identity:

\[
\quad \left\{ \hspace{5pt}
\parbox{5.00in}{
  \noindent Given $f \in L^2(\Omega)$, find $\sol \in \fespace$ such that
  \begin{equation*}
    a(\sol,\testfun) = \ell(\testfun)
  \end{equation*}
  for every $\testfun \in \fespace$, where $a \colon \fespace \times \fespace \rightarrow \Re$ is the bilinear form defined by
  \begin{equation*}
    a(\trialfun,\testfun) = \int_\Omega \trialfun'' \testfun'' \ d \omega
  \end{equation*}
  and $\ell \colon \fespace \rightarrow \Re$ is the linear form defined by
  \begin{equation*}
    \ell(\testfun) = \int_\Omega f \testfun \ d \Omega
  \end{equation*}
  for all $\trialfun, \testfun \in \fespace$.
  }
\right.
\]

To discretize this variational form, we use $H^2$-conforming B-spline finite elements of order $k \ge 4$. (We do not consider splines of order $k = 3$ because, although they are smooth enough for this variational form, the jump discontinuities in the second derivative across quadratic spline elements hinders the optimal convergence rates in the $L^2$ norm \cite[Sec.~2.2]{Strang2008}.) A set of $n$ univariate B-spline basis functions of order $k$, $\{ B^k_i \}_{i=1}^n$ is defined by first providing a knot vector $\Xi = \left\{ \1_1, \1_2, \ldots, \1_{n+k} \right\}$, where $\1_1 = 0$, $\1_{n+k} = 1$, and $\1_i \le \1_{i+1}, \, i = 1,2,\ldots,n+k-1$,. To facilitate strong enforcement of the Dirichlet boundary conditions, we use an open knot vector, i.e., a knot vector with the first and last knots repeated $k$ times: $\1_1 = \ldots = \1_{k} = 0$ and $\1_{n+1} = \ldots = \1_{n+k} = 1$. The interior knots are distinct and, in fact, uniformly spaced. The Cox-de Boor recursion formula given below for $i = 1,2,\ldots, n$ uses this open knot vector to define the univariate B-spline basis for intermediate $\1 \in (0,1)$.
\begin{equation*}
    B^{k}_i(\1) = \frac{\1 - \1_i}{\1_{i+k} - \1_i} B^{k-1}_i(\1) + \frac{\1_{i+k+1}-\1}{\1_{i+k+1}-\1_{i+1}}B_{i+1}^{k-1}(\1), \,    B_i^0(\1) = \left\{ \begin{array}{rl}
    1, & \1 \in [\1_i, \1_{i+1})\\
    0, & \text{elsewhere} \end{array} \right.
\end{equation*}
For notational ease, we henceforth drop the superscript in $B^k$ that denotes the explicit $k$-dependence on the B-spline basis.

B-spline $h$-refinement is done by {\em knot insertion}, where new equi-spaced interior knots $ \tfrac{1}{2}(\1_i + \1_{i+1}), \, k-1 < i < n,$ are added to the original knot vector and the new set of basis functions are computed accordingly. Note that although knot insertion affects neighboring basis functions, it is still a relatively local process, which a variety of knot-insertion strategies exploit. Knot insertion also enables direct construction of prolongation and restriction operators that are naturally transposes of each other, and together with the system matrices they satisfy the Galerkin condition. See \cite{Piegl2012} for a discussion on spline basis functions and relevant algorithms.

The B-spline basis functions allow us to define the finite-dimensional trial space and test space, $\fespace_h \subset \fespace$, used for our Galerkin discretization:
\begin{equation*}
\fespace_h = \left\{ \trialfun_h \in \fespace \Big| \trialfun_h = \sum_i \trialfun_i B_i(\1) \right\} ,
\end{equation*}
where $\sol_i$ are the so-called control points. The discrete variational form of the ODE is then expressed as
$$
\quad \left\{ \hspace{5pt}
\parbox{5.00in}{
  \noindent Given $f \in L^2(\Omega)$, find $\sol_h \in \fespace_h$ such that
  \begin{equation*}
    a(\sol_h,\testfun_h) = \ell(\testfun_h)
  \end{equation*}
  for every $\testfun_h \in \fespace_h$.
  }
\right.
$$
Obtaining the discrete solution amounts to solving linear system \eqref{axb} with $A_{ij} = a(B_j,B_i)$, $\1_i = \sol_i$, and $b_i = \ell(B_i)$.

To assess the efficacy of iterative refinement, we consider the exact solution field:
\begin{equation}
  u = 1-\cos 2 \pi \1.
\label{eq:exactsolution}
\end{equation}
The corresponding forcing function is easily obtained by applying the differential operator, yielding:
\begin{equation*}
  f = -16 \pi^4 \cos 2 \pi \1.
\end{equation*}
This forcing function along with the known solution field $\sol$ enable an exact measure of the total error in our numerical results. We have experimented with other solution fields, but not found anything leading to different conclusions than what we present below.


\section{Numerical Experiments}
\label{sec:results}

We validate the theory here by studying convergence under grid refinement for the model problem. Accordingly, the multigrid solvers we study coarsen only in the mesh size as opposed to the degree of the basis functions. We use Matlab R2019a and the Advanpix toolbox for our experiments. The Advanpix toolbox allows for variable precision computations, although the interface only allows the number of \emph{decimal} digits, $d$, to be specified. 
For our ``exact'' computations, we use $34$ decimal digits of precision, which corresponds to $113$ bits. While slightly more than the 112 bits in IEEE quad precision, we nevertheless refer to $d = 34$ as ``quad precision'' in what follows. For this precision level, Advanpix provides $15$ bits for the exponent, which is consistent with IEEE quad precision. For \emph{all} other levels of precision, Advanpix provides $64$ bits for the exponent. The numerical results reported here are therefore unaffected by the limited dynamic range typically encountered in low precision environments. This aspect of the precision environment corresponds with the theory, which assumes that all computations stay in the dynamical range. Throughout, we use $\textrm{FP16}=2^{-11}$, $\textrm{FP32} = 2^{-24}$, $\textrm{FP64}=2^{-53}$, and $\textrm{FP128}=2^{-112}$ to denote one unit in last place (ulp) for half, single, double, and quad precision, respectively. 

All stiffness matrices and forcing vectors are formed and assembled in quad precision and, as such, are susceptible to quantization (and other) errors at this precision level. The exact solution $x_h$ associated with $u_h$ for a given mesh size $h$ is computed using $A_h$ and $b_h$ formed in quad precision followed by a solve in quad precision. (While this solution is not truly exact of course, its error is insignificant compared to the other errors we consider.) All exact solutions are obtained using Matlab's direct solver with quad precision. The stiffness matrices and forcing vectors pertaining to lower-precision quantizations are obtained by simply rounding their quad-precision counterparts to the desired precision, $\eweq$. We compute $\check{x}_h$ from the lower-precision coefficients by first adding trailing zeros to all numbers to extend them back to quad precision and then solving the resulting system in quad precision. The algebraic solution, $\tilde{x}_h$, is obtained as the solution of $\check{A}_h x_h = \check{b}_h$, where the solvers and precision levels used are specified below for each experiment. Given the true solution, $u$, from (\ref{eq:exactsolution}), along with $x_h$, $\check{x}_h$, and $\tilde{x}_h$, we evaluate the energy norm of the errors shown in (\ref{errdecomp}) in quad precision using $k^2$ quadrature points per element. For familiarity's sake, we use $p=k-1$ to refer to the polynomial degree of the finite element basis functions.

\begin{figure}
\centering
\includegraphics{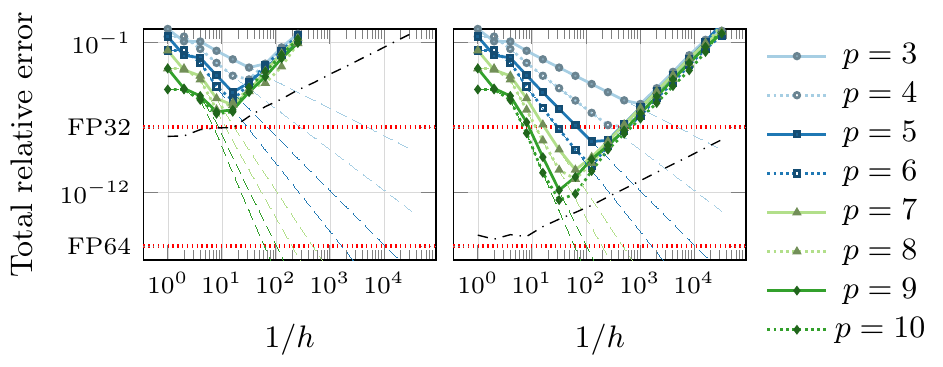}
\caption{The total error in the energy norm when using FMG in fixed precision with $7$ digits (left) and $15$ digits (right) corresponding roughly to single and double precision, respectively. Additionally, the discretization error for each polynomial degree is shown using dashed lines, while the black dash-dotted line shows $\|\fl(x_h)-x_h\|_{A_h}$. This latter quantity is the smallest error one can hope to achieve because it is obtained by simply rounding the exact solution to the chosen precision.}
\label{fig:totalerror}
\end{figure}

\begin{figure}
\centering  
\includegraphics{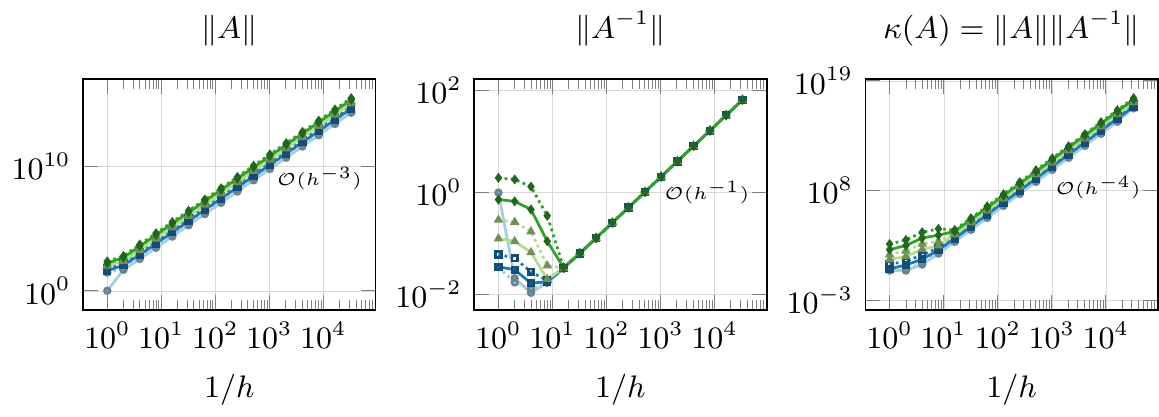}
\caption{Asymptotically, the condition number grows as $\bigO(h^{-4})$ as expected except for the first $5$ levels. For these coarse levels, $\|A^{-1}\|$ shows a distinct pre-asymptotic behavior due to the influence of the boundary conditions. Each plot shows curves for polynomial degrees from $p=3$ (bottom) to $p=10$ (top). The legend is the same as for Figure~\ref{fig:totalerror}.}
\label{fig:normA}
\end{figure}

We begin by confirming in Figure~\ref{fig:totalerror} that {\fmg } in fixed precision is susceptible to multiple types of errors that prevent it from obtaining discritization-error accuracy. For reference, we also show the error from simply rounding the exact solution to the available fixed precision. This is the smallest error we can hope to achieve for a given fixed precision, but {\fmg } is clearly unable to achieve this level of accuracy except possibly for very high order basis functions.

Next, we graph the norm and condition number of $A_h$ in Figure~\ref{fig:normA}. The important observation is that while the asymptotic behavior follows the theory, a pre-asymptotic region exists where the boundaries influence the results. As a consequence, we generally do not expect to see optimal results until after level $4$. We also note that $\|A\|$ clearly depends on $p$, which also carries into $\kappa(A)$. More specifically, $\|A\|$ is approximately proportional to $p^3h^{-3}$. Some of the other quantities that show up in the theory are given by $\kappa(P^tP) = 2^p$, $m_A = 2p+1$, and $m_P = p+2$ for all $h$ past the pre-asymptotic region.

\begin{figure}
\centering  
\includegraphics{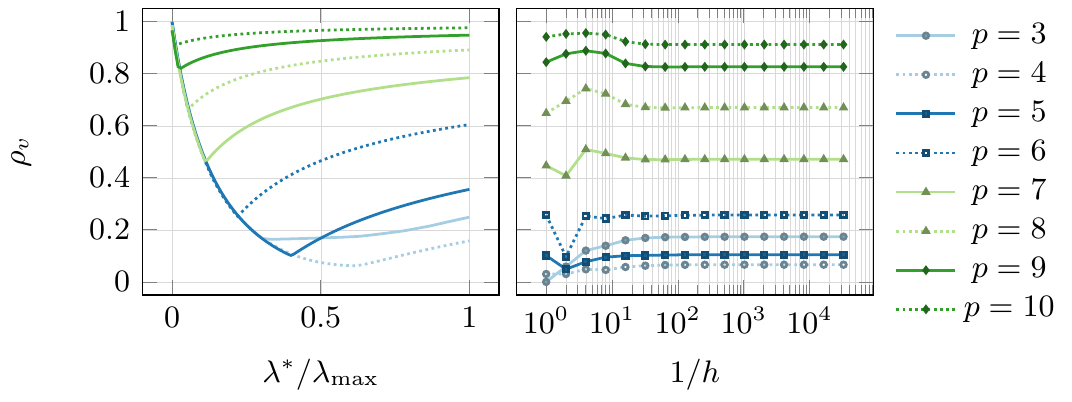}
\caption{The convergence rate of a $2^{\textrm{nd}}$-order Chebyshev smoother as a function of  percentage of the full spectrum that is targeted (left). Given the optimal fraction of the spectrum to target, the right plot shows the convergence rate as a function of mesh size. All computations are done using $34$ digits. For the left plot, $h=1/64$. The curves for $p=3$ are slight outliers because the coarsest level in this case contains no degrees of freedom after the boundary conditions have been imposed.}
\label{fig:Cheb-convergence-1d}
\end{figure}

Throughout, we use V$(1,0)$-cycles and a $2^{\textrm{nd}}$-order Chebyshev smoother based on $\check{D}_h^{-1}\check{A}_h$, where $\check{D}_h$ is the diagonal of $\check{A}_h$. For simplicity, our initial experiments do \emph{not} use progressive precision. On the coarsest level consisting of a single element, we use a single sweep of the smoother. The cost of  a direct solver at that level is insignificant, but also not necessary. 
The Chebyshev smoother depends on knowing the largest eigenvalue of $\check{D}_h^{-1}\check{A}_h$ as well as what percentage of the spectrum to target. We compute the largest eigenvalue using Matlab's standard \texttt{eigs}-function applied to $\check{A}_h$ in double precision. For each polynomial degree, we then determine the lower end of the spectrum, $\lambda^*$, to target by evaluating the convergence rate for a range of different values and picking the one that produces the smallest $\rv$. The results are shown in Figure~\ref{fig:Cheb-convergence-1d}. Clearly, a $2^{\textrm{nd}}$-order Chebyshev smoother is not very effective for high polynomial degrees, but designing effective smoothers is not our aim here, and our theory does not depend on the quality of the smoother.

The convergence rate, $\rv$, is computed in two steps. First, the error propagation matrix, $V$, is constructed column by column by applying one V-cycle with a zero initial guess to each of the canonical basis vectors. Next, $\rv = \|V\|_A$ is computed as the square root of the largest generalized eigenvalue of $V^TAV x = \lambda Ax$. All of this is done in quad precision. We could have approximated $\rv$ by solving $Ax=0$ with a random initial guess $x^{(0)}$ and choosing the largest value over many iterations $i$ of $\|x^{(i+1)}\|_A / \|x^{(i)}\|_A$. However, our eigenvalue approach determines the worst case for the energy convergence rate more effectively.

Next, we study the algebraic error for {\ir-\V} with and without mixed precision. For simplicity, we set $\ewed=\ewe$ because it allows us to isolate the effects of the other precision levels. Figure~\ref{fig:algerror-1d} shows that the error is initially dominated by iteration error before eventually being dominated by rounding error. In fixed-precision, the rounding error never stabilizes, which illustrates why it can be difficult to develop a reliable stopping criterion, but this is much less of an issue in mixed precision. It is also evident that the limiting accuracy, $\chi$, depends on the number of levels in the hierarchy. This is illustrated further in Figure~\ref{fig:MG-Cheb-convergence-1d} (left and middle), where the relative algebraic error after $1000$ V-cycles is shown as a function of $1/h$. As theory predicts, the relative algebraic error grows faster for fixed than for mixed precision, which illustrates the benefit of mixed-precision {\ir-\V}. By comparing algebraic and discretization errors, Figure~\ref{fig:MG-Cheb-convergence-1d} also confirms that, in the absence of progressive precision, multigrid is ultimately dominated by rounding errors.

\begin{figure}
\centering  
 \includegraphics{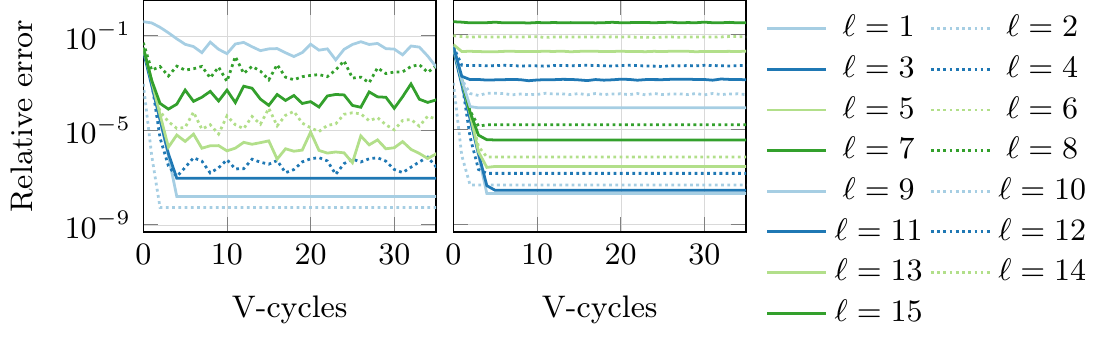}
   \caption{The relative algebraic error in the energy norm, i.e., $e_{\textrm{iter}} + e_{\textrm{round}}$, when solving $\check{A}_hx_h=\check{b}_h$ for $p=4$ with a random initial guess using {\ir-\V } and $1$ to $15$ levels in the hierarchy. The figure on the left illustrates fixed precision using $7$ decimal digits, while the figure on the right illustrates mixed precision using $7$ and $34$ decimal digits for $\ewe$ and $\eweb$, respectively. For both of these experiments, $\ewed=\ewe$ and $\eweq = \eweb$. Initially, $e_{\textrm{iter}}$ dominates until the limiting accuracy is reached. Furthermore, the error generally increases with the grid resolution, and for $\ell > 9$ the fixed precision solver fails to converge.}
  \label{fig:algerror-1d}
\end{figure}

\begin{figure}
\centering  
\includegraphics{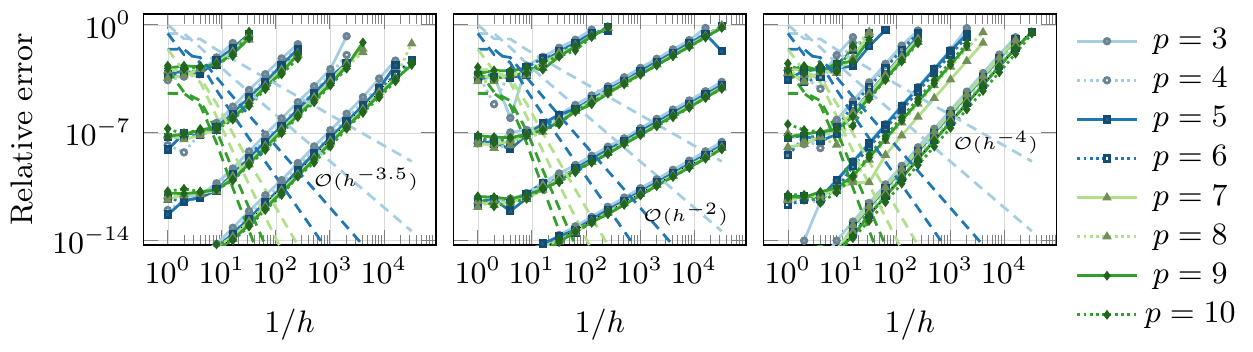}
\caption{The relative rounding error in the energy norm when using fixed precision (left) and mixed precision (middle) as well as the relative quantization error (right). In all three plots, the energy norm of the relative discretization error is shown using dashed lines. The four groups of solid and dotted lines (top to bottom) in each plot correspond to using $3$, $7$, $11$, and $15$ digits for $\eweb$ (left), $\ewe$ (middle), and $\eweq$ (right). For fixed precision (left), $\ewe=\eweb=\eweq$, while for mixed precision (middle), $\eweb$ and $\eweq$ both use $34$ digits. Referring back to Section~\ref{sec:errordefs} these results are compared to $\check{x}_h$ which is computed as described earlier in this section. Finally, for the quantization error experiment (right), $\ewe = \eweb = \eweq$, but the results are compared to the true solution, $x_h$. All the precisions are chosen to ensure that the error is dominated by the choice of $\eweb$, $\ewe$, and $\eweq$, respectively. All plots show the max relative error over the last $50$ iterations when solving the model problem using $1000$ V-cycles. For high precisions and higher polynomial degrees (for which the convergence rate deteriorates), more V-cycles would be necessary to recover the true error.}
\label{fig:MG-Cheb-convergence-1d}
\end{figure}

\begin{figure}
\centering  
\includegraphics{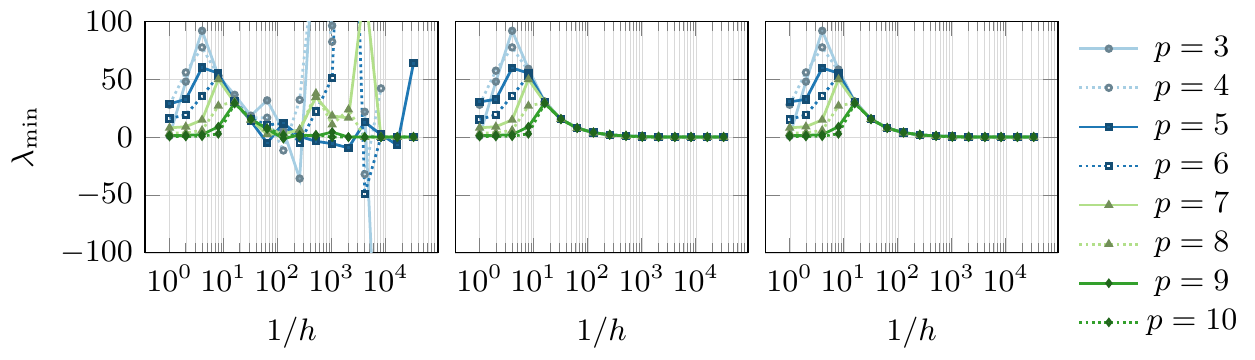}
\caption{The smallest eigenvalue of $\check{A}_h$ when $\check{A}_h$ is quantized to $\ewed$ (left), $\ewe$ (middle), and $\eweq$ (right) precision. For $\ewed$-precision it is clear that $\check{A}_h$ becomes indefinite for fine grid levels, and it should be noted that the smallest eigenvalue can be orders of magnitude below zero, which means that no small diagonal pertubation is likely to recover definiteness. However, it should also be noted that we do not encounter any indefinite matrices in $\ewe$ or $\eweq$ precision. Thus, we can reasonably estimate $\kappa(A_j)$ in $\ewe$ precision using the Lanczos method, for example. }
\label{fig:smallestProgressiveEigenvalue}
\end{figure}

As predicted by our theory, the growth of rounding error for mixed precision is proportional to $\kappa^{1/2}(A)$, or equivalently $\bigO(h^{-2})$, while the observed growth for fixed precision is $\bigO(h^{-3.5})$, which is slightly better than the rate predicted by theory. 
Also shown in Figure~\ref{fig:MG-Cheb-convergence-1d} is the quantization error obtained by solving $\check{A}_h x_h = \check{b}_h$ ``exactly'' for various $\eweq$ and comparing the result to $u_h$. As predicted by Theorem~\ref{abround}, this error grows as $\bigO(h^{-4})$. 

In Figure~\ref{fig:smallestProgressiveEigenvalue}, we confirm that quantization of $A_h$ to $\ewed$-precision can cause it to become indefinite and it can, in fact, become very indefinite for fine grids. 

To implement progressive precision FMG, we need to establish the precisions used at each level, which requires estimating the values for the constants $C$, $c$, $\bar{c}$, $\check{c}$, $\dot{c}$ discussed in Section~\ref{sec:precision}. The choice of $\dot{c}$ was discussed in Section~\ref{sec:precision}, while the values for $c$, $\bar{c}$, and $\check{c}$ can all be estimated based on the data shown in Figure~\ref{fig:MG-Cheb-convergence-1d}. In Section~\ref{sec:precision}, we established bounds for $e_{\textrm{round},\ewe}$, $e_{\textrm{round},\eweb}$, and $e_{\textrm{quant}}$. Here, we treat those expressions as strict equalities to account for the worst case, which yields $e_{\textrm{round},\ewe} = c\ewe h^{-m}$, $e_{\textrm{round},\eweb} = \bar{c}\ewe h^{-2m}$, and $e_{\textrm{quant}} = \check{c}\eweq h^{-2m}$. Generically, and with a slight abuse of notation, this gives us $e = c \ewe h^{-\alpha}$, where $e$ is one of the errors and $c$, $\ewe$, and $\alpha$ are the corresponding constant, precision, and exponent, respectively. It then follows that $\log(e/\ewe) = -\alpha\log(h) + \log(c)$. From this expression, we can compute a linear least squares estimate for $\log(c)$ and $\alpha$ using the data points in Figure~\ref{fig:MG-Cheb-convergence-1d} past the pre-asymptotic region (which in practice we take to be where $1/h > 4$). While Figure~\ref{fig:MG-Cheb-convergence-1d} only shows data for $4$ different precisions, we have conducted the experiments for all precisions between 3 and 15 decimal digits, and we use the data from all the experiments for the least squares estimates except that we omit the data from the pre-asymptotic region ($1/h\leq 16$). The estimates for $c$, $\bar{c}$, and $\check{c}$ are shown in Figure~\ref{fig:constants}.

Unfortunately, it is computationally quite expensive to obtain all the data required for these least squares estimates. As an alternative, given $c_\kappa \geq \kappa h^{2m}$, we can estimate all the constants quite cheaply by noting from Section~\ref{sec:precision} that $\check{c} = c_\kappa$, $\bar{c} \lesssim 4\alabar c_\kappa$, and $c = \sqrt{c_\kappa}$. Furthermore, from Figure~\ref{fig:normA}, we see that this can be estimated reliably as soon as we get past the pre-asymptotic region. In practice, we therefore only have to estimate the condition number for a few small matrices. Technically, we have a lower bound for $c_\kappa$ that is quite a bit higher in the pre-asymptotic region. However, the bounds for $c$, $\bar{c}$, and $\check{c}$ based on $c_\kappa$ are rather conservative to begin with, so we find in practice that it is safe to ignore this technicality and use the asymptotic value of $c_\kappa$ for all $h$. In fact, Figure~\ref{fig:constants} shows that the constants obtained using $c_\kappa$ can be several orders of magnitude larger than the least squares estimates. This may seem concerning, but each order of magnitude translates to using one additional decimal digit for the corresponding precision level, and this fixed amount of extra precision is relatively insignificant for the higher levels that tend to account for most of the computational cost.
The entire approach for computing the constants is captured in Algorithm~\ref{alg-constants}. Also included is the computation of $\I$ based on (\ref{eN}), with the results shown in Table~\ref{tab:FMGiterations}. 

\begin{figure}
\centering  
\includegraphics{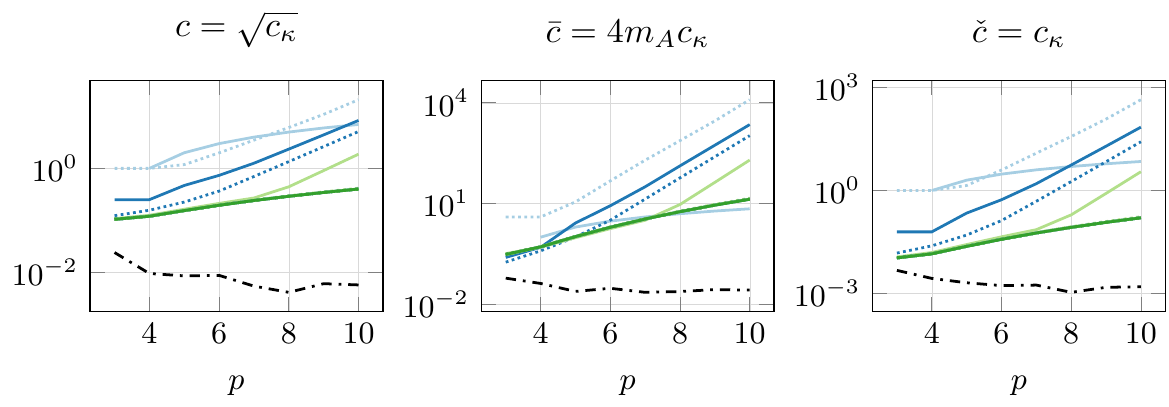}
\caption{Estimates for $c$, $\bar{c}$, and $\check{c}$. The legend here is the same as for Figure~\ref{fig:algerror-1d} with the colored lines representing estimates based on $c_\kappa$. The dash-dotted line in black is the linear least squares estimate based on the data partially shown in Figure~\ref{fig:MG-Cheb-convergence-1d}. The graphs here suggest that the true constants depend rather weakly (and inversely) on $p$, while the estimates based on $c_\kappa$ suggest a significant growth with $p$. }
\label{fig:constants}
\end{figure}

\begin{table}[h]
  \centering
  \begin{tabular}{ccccccccc}
  p                     & 3 & 4 & 5 & 6 & 7 &  8  &  9  & 10 \\
  \hline
  Theoretical N  & 2 & 2 & 2 & 4 & 8 & 17 & 38 & 85 \\
  Minimal N       & 1 & 1 & 1 & 2 & 4 &  9  & 28 & 50
  \end{tabular}
  \vspace{1em}
  \caption{Number of V-cycles required inside {\fmg } as a function of the polynomial degree, according to the theory in (\ref{eN}) and given the convergence rates obtained in Fig.~\ref{fig:Cheb-convergence-1d}. Also shown is the smallest number of V-cycles for which {\fmg } actually converges when using the constants estimated from $c_\kappa$. This shows that the theory is somewhat conservative, but mostly for high polynomial degrees where the convergence rate of the smoother is poor. The minimal number of V-cycles increases by one in a few cases if the smaller constants obtained from least squares estimation are used instead. }
  \label{tab:FMGiterations}
  \end{table}

It remains to estimate $C$. Given the discretization error as plotted in Figures~\ref{fig:totalerror} and \ref{fig:MG-Cheb-convergence-1d}, $C$ can easily be obtained by linear regression. However, those curves are based on computations in exact arithmetic and knowledge of the exact solution. Fortunately, we can estimate $C$ in the course of running {\fmg } based on the strong approximation property in (\ref{C}). Ultimately that leads us to the progressive FMG algorithm outlined in Algorithm~\ref{alg-pfmg}. Developing all the details to deal robustly with any pre-asymptotic region is beyond the scope of this paper. Still, this algorithm is notable by starting out in low precision and only advancing to higher precision as necessary in order to achieve the specified error goal. 

\noindent
\begin{center}
\begin{minipage}{0.97\linewidth}
\begin{algorithm}[H]
\footnotesize
\caption{Compute constants for progressive FMG}
\label{alg-constants}
\begin{algorithmic}[1]
\Require $A$, $p$, $m$, $\theta$, $\text{tol} < 1$, $\dot{\tau}_\text{tol} < 1$.
\State $q \leftarrow p+1-m$ \COMMENT{Compute $q$}
\State $\kappa_0 \leftarrow \| A_0 \| \| A_0^{-1} \|$ \COMMENT{Compute condition number of $A_0$}
\State $j \leftarrow 0$ \COMMENT{Initialize level counter}
\Repeat
\State $j \leftarrow j + 1$ \COMMENT{Update level counter}
\State $\kappa_j \leftarrow \| A_j \| \| A_j^{-1} \|$ \COMMENT{Compute condition number of $A_j$}
\Until {$\left| \frac{\kappa_j}{\kappa_{j-1} } \theta^{-2m} - 1 \right| < \text{tol}$} \COMMENT{Stop if in asymptotic region}
\State $c_\kappa \leftarrow \kappa_j \theta^{-2mj}$ \COMMENT{Compute $c_\kappa$}
\State $c \leftarrow \sqrt{c_{\kappa} }$ \COMMENT{Compute $c$}
\State $\bar{c} \leftarrow 4 m_A c_{\kappa} $ \COMMENT{Compute $\bar{c}$}
\State $\check{c} \leftarrow c_{\kappa} $ \COMMENT{Compute $\check{c}$}
\State $\dot{c} \leftarrow \dot{\tau}_\text{tol}/\sqrt{c_{\kappa}}$ \COMMENT{Compute $\dot{c}$}
\State Compute $\rho$ for level $j$ \COMMENT{Determine asymptotic convergence factor.}
\State $\I \leftarrow (\log_2(5) + q\log_2(\theta))/(| \log_2(\rho)|)$ \COMMENT{Compute theoretical number of V-cycles}
\State\Return $(c, \bar{c}, \check{c}, \dot{c},\I)$\COMMENT{Return constants}
\end{algorithmic}
\end{algorithm}
\end{minipage}
\end{center}
\vspace{1em}

\noindent
\begin{center}
\begin{minipage}{0.97\linewidth}
\begin{algorithm}[H]
\footnotesize
\caption{$\texttt{Progressive FMG$(1,0)$-Cycle (\pfmg)}$}
\label{alg-pfmg}
\begin{algorithmic}[1]
  \Require $\mathcal{L}$, $f$, $m$, $k$, $\theta$, $c$, $\bar{c}$, $\check{c}$, $\dot{c}$, $\I \ge 1$, $e_\text{goal} < 1$
  \State $q \leftarrow k-m$
  \State $x_0 \leftarrow 0$
  \State $j \leftarrow 1$
  \Loop
  \State $($\textcolor{med}{$\ewe$},\textcolor{hi}{$\eweb$},\textcolor{orange}{$\eweq$},\textcolor{lo}{$\ewed$}$) \leftarrow $ComputePrecisions$(c, \bar{c}, \check{c}, \dot{c},j)$ \COMMENT{Update all precision levels}
  \State \textcolor{orange}{$(\check{A}_j,\check{b}_j,\check{P}_j,h_j) \leftarrow$ Discretize$(\mathcal{L},f,k,j)$}\COMMENT{\textcolor{orange}{Discretize PDE at level $j$}}
  \State \textcolor{med}{$x_j \leftarrow \check{P}_jx_{j-1}$}\COMMENT{\textcolor{med}{Interpolate from previous level}}

  \State $i \leftarrow 0$\COMMENT{Initialize {\ir } }
  \While{$i < \I$}
  \State \textcolor{med}{$r_j \leftarrow$}\textcolor{red}{$\check{A}_jx_j - \check{b}_j$} \COMMENT{\textcolor{red}{Update {\ir } residual} and \textcolor{med}{round}}
  \State \textcolor{lo}{$y_j \leftarrow $\V$(\check{A}_j, r_j, \check{P}_j, j)$}\COMMENT{\textcolor{lo}{Compute correction by \V}}
  \State \textcolor{med}{$x_j \leftarrow x_j - y_j$}\COMMENT{\textcolor{med}{Update approximate solution of $A_jx_j = b_j$}}
  \State $i \leftarrow i + 1$\COMMENT{Increment {\ir } cycle counter}
  \EndWhile
  \If{$j>4$}\COMMENT{In asymptotic region ?}
  \State \textcolor{med}{$C \leftarrow \frac{\| \check{P}_j x_{j-1} - x_j \|_{\check{A}_j}}{h_{j-1}^q \|x_j\|_{\check{A}_j}}$}\COMMENT{\textcolor{med}{Estimate discretization constant}}
  \State \textcolor{med}{$\ell \leftarrow \left\lceil \frac{1}{q} \log_\theta \left( \frac{C}{e_\text{goal}} \right) \right\rceil$}\COMMENT{\textcolor{med}{Compute required number of levels}}
  \If{$\ell\le j$}
  \State\Return $x_j$ \COMMENT{Return solution with error less than $e_\text{goal}$}
  \EndIf
  \EndIf
  \State $j \leftarrow j+1$
  \EndLoop
\end{algorithmic}
\end{algorithm}
\end{minipage}
\end{center}
\vspace{1em}

Using the proposed algorithm, the precision requirements and the accuracy actually achieved is shown in Figure~\ref{fig:fmg-precisions-accuracy-1d}. Most importantly, we observe that {\pfmg} does in fact achieve discretization-error accuracy. However, we also note that the use of standard floating point types available in hardware can be surprisingly restrictive in terms of $h$. In Figure~\ref{fig:fmg-precisions-second-order}, we extrapolate the results to second-order PDEs since these are quite common in real applications. For high-order basis functions, the order of the PDE does not matter much, and we notice that while there is a difference between $\eweb$ and $\eweq$, it is relatively insignificant in this regime. For lower-order basis functions, a final observation is that $\eweq \ll \ewe$, meaning that $\check{A}_h$ used in the residual computation in {\ir } must be of sufficiently high precision.

\begin{figure}
\centering
  \includegraphics{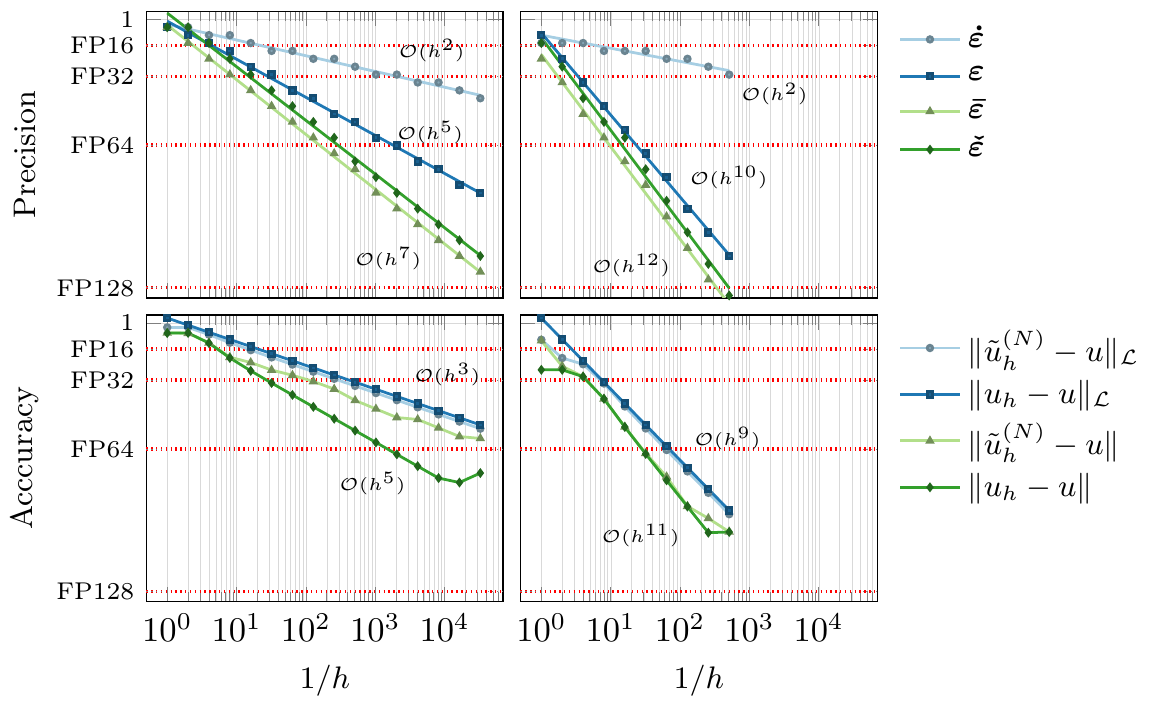}
  \caption{Precision requirements for progressive precision FMG in order to reach discretization error accuracy for the model problem (top), and the actual accuracy obtained compared to the true discretization error (bottom). The graphs shown here are for $p=4$ (left) and $p=10$ (right). For reference, we include the $L^2$ error in the accuracy plots, and notice that in general we do not achieve optimal convergence in the $L^2$-norm. However, for $p=10$ the large number of V-cycles we use due to the conservative nature of the estimate for $\I$ probably accounts for achieving close to optimal results in the $L^2$-norm too. }
  \label{fig:fmg-precisions-accuracy-1d}
\end{figure}

\begin{figure}
\centering
\includegraphics{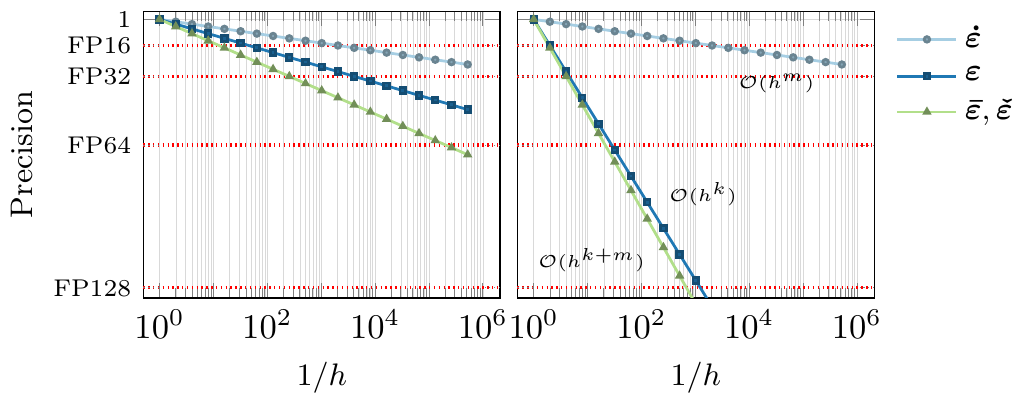}
  \caption{Predicted precision requirements for progressive precision FMG for a second-order PDE (assuming that $C=\dot{c}=c=\bar{c}=\check{c}=1$ for simplicity). The graphs shown here are for $p=1$ (left) and $p=10$ (right). For $p=1$, double precision suffices up to $1/h=2^{18}$, while, for $p=10$, anything beyond $1/h=2^{4}$ is contaminated by quantization errors when using double precision. Since quantization errors do not depend on the choice of solver, these limits apply to any kind of solver and not just FMG. Notice also that these limits on $h$ apply to problems in all dimensions.}
  \label{fig:fmg-precisions-second-order}
\end{figure}

\section{Conclusions}
\label{sec:conclusion}
This paper has successfully shown the potential of using progressive precision multigrid methods for solving linear elliptic PDEs to arbitrary accuracy given sufficient but parsimoniously chosen precisions in all computations. Compared to existing work, the key to this success on one hand is the observation that quantization errors play a critical role that must be accounted for. On the other hand is the observation that the V-cycle is very resilient and will work correctly even when it is being run in such low precision that the matrices involved may become indefinite simply from rounding them to working precision. The limitations introduced by quantization error ultimately lead to fairly strict limitations on the grid size that can be used to discretize the PDE for any given precision budget. This is worth noting because many computations in practice are limited to standard IEEE double precision at the high end. Insofar as the PDE solution is sufficiently smooth, higher-order elements generally allow for higher accuracy. However, when the grid size restriction is taken into consideration, the improvement for a given maximum precision is relatively small. 

In order to choose all the precision levels, we have introduced a heuristic that balances all the different types of errors. This approach ensures that we avoid ``overcomputation'', where one type of error is reduced only to be swamped by some other type of error. Assuming that one has appropriate bounds, this idea can easily be generalized to include other types of errors such as those from matrix assembly or even modeling errors. Given an appropriate performance model, it can also be generalized to account for different costs associated with different types of errors. Both of these extensions are interesting topics for future work.
Other topics for future work include the extension of the ideas presented here to algebraic multigrid, and a proper analysis of any effects due to overflow or underflow. 

\bibliographystyle{siamplain}
\bibliography{mpmg}

\end{document}